\documentclass[11pt]{article}
\usepackage[ansinew]{inputenc}
\usepackage[bbgreekl]{mathbbol}
\usepackage{amsmath,amssymb}
\usepackage{graphicx}

\usepackage{url}
\usepackage[numbers,square]{natbib}
\usepackage{multirow}
\usepackage{tikz}
\usepackage{tikz-cd}
\usetikzlibrary{arrows,matrix}
\usepackage{bbm}
\usepackage{authblk}
\usepackage{subcaption}

\usepackage{stmaryrd}
\usepackage{comment}
\usepackage{soul}

\usetikzlibrary{calc,patterns,shapes.geometric}
\usetikzlibrary{arrows,shapes,positioning}
\usetikzlibrary{decorations.markings,decorations.pathmorphing,decorations.pathreplacing}
\usetikzlibrary{decorations.text,decorations.shapes}

\usepackage[bbgreekl]{mathbbol}
\usepackage{amsmath,amssymb}
\usepackage{bm}
\usepackage{bbm}
\usepackage{upgreek}

\DeclareFontFamily{OT1}{pzc}{}
\DeclareFontShape{OT1}{pzc}{m}{it}{<-> s * [1.10] pzcmi7t}{}
\DeclareMathAlphabet{\mathpzc}{OT1}{pzc}{m}{it}

\newcommand{\cli}{\mathpzc{i}}
\newcommand{\clr}{\mathpzc{r}}

\def\mint{\int \mspace{-21.mu}\raise .8ex\hbox{\rotatebox{-80}{$|$}\,}}
\def\smint{\smallint \mspace{-13.mu}\raise -.1ex\hbox{\rotatebox{10}{--}}}

\def\beq{\begin{equation}}
\def\eeq{\end{equation}}

\def\<{\langle}
\def\>{\rangle}

\def\ie{{\em i.e. }}

\def\Chi{\raise .3ex \hbox{\large $\chi$}} 

\def\dd{\, {\rm d}} 

\def\tb{\hbox{$\|\kern -.09em |$}}
\def\bigtb{\hbox{$\big\|\kern -.09em \big|$}}
\def\Bigtb{\hbox{$\Big\|\kern -.09em \Big|$}}

\def\gx{\bx}  

\def\NN{\mathbbm{N}}

\def\RR{\mathbbm{R}}

\newcommand{\mat}[1]{\mathbb{#1}}
\def\matd{\mathbb{d}}
\def\matr{\mathbb{r}}

\def\matC{\mathbb{C}}
\def\matD{\mathbb{D}}

\def\matG{\mathbb{G}}
\def\matH{\mathbb{H}}
\def\matrh{\mathbb{h}}
\def\matI{\mathbb{I}}
\def\matJ{\mathbb{J}}

\def\matO{\mathbb{O}}
\def\matR{\mathbb{R}}
\def\matS{\mathbb{S}}

\def\matW{\mathbb{W}}

\newcommand\arrg{\textbf{\textsf{g}}}
\newcommand\arrA{{\bm{\mathsf{A}}}}
\newcommand\arrB{{\bm{\mathsf{B}}}}

\newcommand\arrD{{\bm{\mathsf{D}}}}
\newcommand\arrE{{\bm{\mathsf{E}}}}
\newcommand\arrF{{\bm{\mathsf{F}}}}
\newcommand\arrH{{\bm{\mathsf{H}}}}
\newcommand\arrf{{\textbf{\textsf{f}}}}
\newcommand\arrG{{\bm{\mathsf{G}}}}
\newcommand\arrJ{{\bm{\mathsf{J}}}}
\newcommand\arrU{{\bm{\mathsf{U}}}}
\newcommand\arrV{{\bm{\mathsf{V}}}}

\newcommand\arrX{{\bm{\mathsf{X}}}}

\newcommand\arrphi{{\bm{\mathsf{\phi}}}}
\newcommand\arrpsi{{\bm{\mathsf{\psi}}}}
\newcommand\arrrho{{\bm{\mathsf{\rho}}}}

\def\cC{\mathcal{C}}

\def\cH{\mathcal{H}}
\def\cI{\mathcal{I}}

\def\cL{\mathcal{L}}

\def\cO{\mathcal{O}}

\def\cR{\mathcal{R}}

\usepackage{mathrsfs}

\newcommand\uvec{{\be}}

\def\bb{{\bs b}}

\def\be{{\bs e}}

\def\bsm{{\bs m}}

\def\bv{{\bs v}}

\def\bx{{\bs x}}

\def\bA{{\bs A}}
\def\bB{{\bs B}}

\def\bD{{\bs D}}
\def\bE{{\bs E}}
\def\bF{{\bs F}}
\def\bG{{\bs G}}
\def\bH{{\bs H}}

\def\bJ{{\bs J}}

\def\bM{{\bs M}}

\def\bV{{\bs V}}

\def\bX{{\bs X}}

\def\bs{\boldsymbol}

\def\dt{{\partial_{t}}}

\DeclareMathOperator{\diag}{diag}

\DeclareMathOperator{\Div}{div}

\DeclareMathOperator{\curl}{curl}

\DeclareMathOperator{\grad}{grad}

\providecommand{\range}[2]{\llbracket #1, #2 \rrbracket}

\newcommand{\fracp}[2]{\ensuremath{\frac{\partial #1}{\partial #2}}}

\newcommand{\half}{\frac{1}{2}}

\newcommand{\hi}{\hat{i}}
\newcommand{\hj}{\hat{j}}
\newcommand{\hk}{\hat{k}}

\newcommand{\hodgeTwo}{\matH_2}
\newcommand{\hodgeThree}{\matH_3}
\newcommand{\hodgeDualTwo}{\tilde{\matH}_2}
\newcommand{\hodgeDualThree}{\tilde{\matH}_3}

\newtheorem{theorem}{Theorem}
\newtheorem{lemma}{Lemma}
\newtheorem{remark}{Remark}
\newtheorem{definition}{Definition}
\newtheorem{assumption}{Assumption}

\newtheorem{proof}{Proof}

\usepackage[normalem]{ulem}

\title{A dual grid geometric electromagnetic particle in cell method}

\author[1]{Katharina Kormann}
\author[2,3]{Eric Sonnendr\"ucker}
\affil[1]{Ruhr University Bochum, Department of Mathematics, Bochum, Germany}
\affil[2]{Max Planck Institute for Plasma Physics, Garching, Germany}
\affil[3]{Technical University of Munich, Department of mathematics, Garching, Germany}
\date{}

\begin{document}

\maketitle

\begin{abstract}
Geometric particle-in-cell discretizations have been derived based on a discretization of the fields that is conforming with the de Rham structure of the Maxwell's equation and a standard particle-in-cell ansatz for the fields by deriving the equations of motion from a discrete action principle. While earlier work has focused on finite element discretization of the fields based on the theory of Finite Element Exterior Calculus, we propose in this article an alternative formulation of the field equations that is based on the ideas conveyed by mimetic finite differences. The needed duality being expressed by the use of staggered grids. We construct a finite difference formulation based on degrees of freedom defined as point values, edge, face and volume integrals on a primal and its dual grid. Compared to the finite element formulation no mass matrix inversion is involved in the formulation of the Maxwell solver. In numerical experiments, we verify the conservation properties of the novel method and study the influence of the various parameters in the discretization.
\end{abstract}


\section{Introduction}

A kinetic description of a plasma in external and self-con\-sist\-ent fields is given by the Vlasov equation for the particle distribution functions coupled to Maxwell's equation. Numerical schemes that preserve the structure of the kinetic equations can provide new insights into the long time behavior of fusion plasmas. 
In \cite{kraus2016gempic, CPKS_variational_2020} we developed a geometric Particle in Cell (PIC) method, where the fields were computed in compatible Finite Element spaces forming a discrete de Rham complex. In this paper, we will introduce a discrete de Rham complex where the unknowns are what we called geometric degrees of freedom, \textit{i.e.} point values, edge integrals, face integrals and volume integrals.
Moreover, whereas the duality inherent in Maxwell's equations, is naturally expressed as a $L^2$ duality for Finite Elements, we will here use a duality based on staggered grids. We consider a primal grid and its dual grid whose vertices are the barycenters of the cells. We restrict ourselves to Cartesian grids in this paper, but the construction could be extended to arbitrary grids.

This yields a generalized Finite Difference formulation that generalizes the classical second-order Yee scheme. 
Here the discrete differential and Hodge operators will be matrices acting on the appropriate degrees of freedom. Moreover, we will need to also define an interpolation operator from the discrete level to the continuous spaces. These constructions are based on the principles of mimetic discretization as presented by Bochev and Hyman \cite{bochev2006principles}.
For this purpose, we will use a representation of the solution based on sliding Lagrange polynomials. Based on this representation of the fields and a standard particle-in-cell discretization of the distribution function, we formulate a semi-discrete Lagrangian from which we derive the semi-discrete equations of motion for the Vlasov--Poisson system from a semi-discrete variational principle. We show that the resulting system preserves the Jacobi-identity, the divergence constraints as well as energy on the semi-discrete level. Finally, we use a Poisson integrator based on a Hamiltonian splitting as proposed in \cite{Crouseilles:2015, Qin:2015, kraus2016gempic} to discretize in time. 
The algorithm was implemented based on the AMReX library \cite{AMReX_JOSS} and numerical illustrations show the good conservation properties of our algorithm.

\section{Vlasov--Maxwell system}

The (nonrelativistic) Vlasov--Maxwell system describes the dynamics of plasma particles in their self-consistent and external electromagnetic fields in a kinetic picture where the particles of one species $s$ are represented by the distribution function $f_s$ in phase space $(\bx,\bv) \in \Omega \times \RR^3$, $\Omega \subset \RR^3$, which evolves according to the Vlasov equation
\begin{equation*}
	\dt f_s(t, \bx, \bv) + \bv \cdot \nabla_\bx f_s(t, \bx, \bv)
    + \frac{q_s}{m_s}\big( \bE(t,\bx) + \bv \times \bB(t,\bx)   \big) \cdot \nabla_\bv f_s(t, \bx, \bv) = 0
\end{equation*}
where $m_s$ and $q_s$ denote the mass and charge of the particle species $s$ and $\bE$, $\bB$ denote the electric and magnetic fields. 
The Vlasov equation forms a hyperbolic conservation law and the distribution function is constant along the characteristic curves satisfying the following system of ordinary differential equations
\begin{equation} \label{traj}
\frac{\dd}{\dd t} \bX(t) = \bV(t),
\qquad
\frac{\dd}{\dd t} \bV(t) =  \frac{q_s}{m_s}\big( \bE(t, \bX(t)) + \bV(t) \times \bB(t, \bX(t))   \big) .
\end{equation}
Since the particles are charged, the dominating force is the Lorentz force governing the velocity advection. The self-consistent part of the electromagnetic fields obey the Maxwell's equations
\begin{equation} \label{Maxwell}
\begin{aligned}
&\dt \bD(t, \bx) = \curl \bH(t, \bx) - \bJ(t, \bx) \\
&\dt \bB(t, \bx) = - \curl \bE(t, \bx)\\
&\Div \bD(t, \bx) = \rho(t, \bx) \\
&\Div \bB(t, \bx) = 0
\end{aligned}
\end{equation}
which are here posed in the four-field formulation where the electric field is described by $\bE$ and $\bD$, and the magnetic field by $\bH$ and $\bB$. The fields $\bD$ and $\bE$ as well as $\bH$ and $\bB$ are related through the constitutive relations of the material, i.e.~$\bD = \varepsilon \bE$, $\bH = \frac{1}{\mu}\bB$, where $\varepsilon$ denotes the permittivity and $\mu$ the permeability. The source terms are given as the charge and current density of the distribution functions
\begin{equation*}
	\rho(t, \bx) = \sum_s q_s \int_{\RR^3} f_s(t, \bx, \bv) \dd \bv,
  \qquad \bJ(t, \bx) = \sum_s q_s \int_{\RR^3} \bv f_s(t, \bx, \bv) \dd \bv.
\end{equation*}
The total energy in the system is given by
\begin{equation*}
	\cH =  \sum_s\frac{m_s}{2}\int_{\Omega} \int f_s(t,\bx,\bv) | \bv|^2 \, \dd \bv \dd \bx + \half \int \bE \cdot \bD \, \dd \bx + \half\int \bB \cdot \bH \, \dd \bx
\end{equation*}
and is conserved for suitable boundary conditions that keep the system closed, e.g. periodic boundary conditions as will be considered in this paper.

Low \cite{Low:1958.royal} formulated the Lagrangian of the system
\begin{equation}\begin{aligned} \label{cL}
  \cL 
  = \sum_s \int f_s(t_0,\bx_0,\bv_0) \left( \big(m_s \bV + q_s \bA(t,\bX)\big)\cdot \dot{\bX} - \Big(\frac{m_s}{2}\bV^2 +q_s\phi(t,\bX)\Big)\right) \dd \bx_0 \dd \bv_0 \\
  +\frac {1}{2} \int_{\Omega} \varepsilon(\bx) |\grad\phi(t,\bx) +\dot{\bA}(t,\bx)|^2 \dd \bx  - \frac {1}{2}\int_{\Omega}\frac{1}{\mu(\bx)} |\curl \bA(t,\bx)|^2  \dd \bx.
\end{aligned}\end{equation}
Here $\bA$ denotes the electromagnetic vector potential and $\bX=\bX(t_0;\bx_0,\bv_0)$, $\bV=\bV(t_0;\bx_0,\bv_0)$ denote the solutions at time $t_0$ of the characteristic equations of motion \eqref{traj} for initial conditions $(\bx_0,\bv_0)$. 
The equations \eqref{traj}--\eqref{Maxwell} can be derived by a variational principle from this Lagrangian. In section \ref{sec:discrete_lagrangian}, we will derive the semi-discrete equations of motion from a variational principle applied to a semi-discretized Lagrangian.

The Maxwell's equations themselves possess a geometric structure which is linked with the following de Rham complex (cf.~\cite{Hiptmair2002, ArnoldFEEC2018}):
\begin{equation*}\label{de Rham}
\begin{tikzpicture}[baseline=(current  bounding  box.center)]
  \matrix (m) [matrix of math nodes,row sep=3em,column sep=4em,minimum width=2em] {
           H^1(\Omega) & H(\curl,\Omega) & H(\Div,\Omega) & L^2(\Omega)
    \\
    };
  \path[-stealth]
    (m-1-1) edge node [above] {$\grad$} (m-1-2)
    (m-1-2) edge node [above] {$\curl$} (m-1-3)
    (m-1-3) edge node [above] {$\Div$}  (m-1-4)
    ;
\end{tikzpicture}
\end{equation*}
Let us consider $\bE,\bD$ as well as $\bB,\bH$ as two different mathematical objects with $\bE, \bH \in H(\curl,\Omega)$ and $\bD, \bB \in H(\Div,\Omega)$. Then, we can consider two de Rham complexes, one for $\bE, \bB$ and one for $\bD,\bH$ and group Maxwell's equation in two pairs where Faraday's law and the magnetic Gauss' law are linked to the first and Ampere's law and the electric Gauss law to the second de Rham complex. We consider the two complexes as dual complexes and relate the corresponding fields by the Hodge operator $\bE = \star \bD$ and $\bH = \star \bB$. This yields the following twin diagram
\begin{equation*}
\begin{tikzpicture}[baseline=(current  bounding  box.center)]
  \matrix (m) [matrix of math nodes,row sep=3em,column sep=3.6em,minimum width=2em] {
          \phi \in H^1(\Omega) & \bE, \bA \in H(\curl,\Omega) & \bB \in  H(\Div,\Omega) & L^2(\Omega)
    \\
               \tilde{\rho} \in  L^2(\Omega) & \tilde{\bD}, \tilde{\bJ} \in  H(\Div,\Omega) & \tilde{\bH} \in  H(\curl,\Omega) & H^1(\Omega)
    \\
    };
  \path[-stealth]
    (m-2-1) edge node [left]  {$\star$} (m-1-1)
    (m-1-1) edge node [above] {$\grad$} (m-1-2)
    (m-2-2) edge node [left]  {$\star$} (m-1-2)
    (m-2-2) edge node [above] {$-\Div$} (m-2-1)
    (m-2-3) edge node [left]  {$\star$} (m-1-3)
    (m-2-4) edge node [left]  {$\star$} (m-1-4)
    (m-1-2) edge node [above] {$\curl$} (m-1-3)
    (m-2-3) edge node [above] {$\curl$} (m-2-2)
    (m-1-3) edge node [above] {$\Div$}  (m-1-4)
    (m-2-4) edge node [above] {$-\grad$}  (m-2-3)
    ;
\end{tikzpicture}
\end{equation*}
Note that we have marked objects on the second complex, later referred to as the \emph{dual complex}, with a tilde as will be done throughout the rest of the paper.

Linked to this structure are the conservation properties of the Maxwell's equation: In particular, the divergence constraints remain satisfied over time (given suitable initial and boundary conditions) due to the following continuity equation, which holds between charge and current density
$$
\partial_t \rho + \nabla \cdot \bJ = 0.
$$
This continuity equation is obtained from the Vlasov equation by integration over velocity space.

\section{Mimetic finite differences}

In the previous section, we have seen that the electromagnetic fields live in functional spaces forming a de Rham diagram. Since the diagram is tightly linked to the conservation properties of the Maxwell's equations, a discretization that mimics the diagram discretely guarantees similar properties in the discrete setting. In particular, discrete spaces should be chosen for the various fields such that they are related to the continuous space through a commuting diagram based on well-chosen projections. The framework of Finite Element Exterior Calculus \cite{Arnold.Falk.Winther.2006.anum, Arnold.Falk.Winther.2010.bams} constructs a commuting diagram between the infinite dimensional solution spaces and the finite dimensional finite-element spaces and implements duality by using strong derivatives in one diagram and weak derivatives in the other. 

This section is devoted to a construction of a commuting de Rham diagram for discrete degrees of freedom to derive an entirely discrete formulation accessible to finite differences. For the discrete formulation, we consider restrictions from the continuous spaces to degrees of freedom. Traditionally for finite differences, point values are used as degrees of freedom. For mimetic finite differences, instead, we consider  four kinds of unknowns related to the geometric quantities on the grid: point values, edge integrals, face integrals and volume integrals. Duality is achieved by defining a primal grid for the first de Rham sequence and defining the degrees of freedom of the dual complex and its dual (that is \emph{staggered}) grid.  We note that these so-called \emph{geometric degrees of freedom} have also been introduced in \cite{CPKS_variational_2020} in the finite element setting as a finite representation of the functions in the finite-element spaces.

Following the conventions proposed in \cite{kreeft2011}, we denote by $\cR_i$, $i=0,1,2,3$, the reduction operators associating grid quantities to fields: $\mathcal{R}_0$ associates to a continuous function its values at all the vertices of the grid, $\mathcal{R}_1$ associates to a vector valued functions the circulations on all the edges, 
$\mathcal{R}_2$ associates to a vector valued function the integrals on all the fluxes through all face,  and
$\mathcal{R}_3$ associates to a scalar function its integrals on all the cells of the grid. 

Denoting by $N_0, N_1, N_2, N_3$ respectively the number of vertices, edges, faces and cells on the primal grid, we denote respectively by $\cC_0 = \mathbb{R}^{N_0}, \cC_1 = \mathbb{R}^{N_1}, \cC_2= \mathbb{R}^{N_2}, \cC_3= \mathbb{R}^{N_3}$ the vectors containing all the vertex, edge, face and cell degrees of freedom. Analogously, we denote the vectors containing the geometric degrees of freedom on the dual grid by $\tilde{\cC}_0 = \mathbb{R}^{\tilde{N}_0}, \tilde{\cC}_1= \mathbb{R}^{\tilde{N}_1}, \tilde{\cC}_2= \mathbb{R}^{\tilde{N}_2}, \tilde{\cC}_3= \mathbb{R}^{\tilde{N}_3}$.  We denote all objects related to the dual space with letters and symbols with tilde. 
With this notation, let us consider the following primal-dual continuous-discrete de Rham diagram:
\begin{equation}\label{CGL}
\begin{tikzpicture}[baseline=(current  bounding  box.center)]
  \matrix (m) [matrix of math nodes,row sep=3em,column sep=4em,minimum width=2em] {
           H^1(\Omega) & H(\curl,\Omega) & H(\Div,\Omega) & L^2(\Omega)
    \\
           \mathcal{C}_0 &  \mathcal{C}_1 &  \mathcal{C}_2 &  \mathcal{C}_3
    \\
           \tilde{\mathcal{C}}_3 &  \tilde{\mathcal{C}}_2 &  \tilde{\mathcal{C}}_1 &  \tilde{\mathcal{C}}_0
    \\
          L^2(\Omega) & H(\Div,\Omega) &  H(\curl,\Omega) & H^1(\Omega)
    \\
    };
  \path[-stealth]
%
    (m-2-1.75) edge node [right] {$\cI_0$} (m-1-1.285)
(m-2-2.75) edge node [right] {$\cI_1$} (m-1-2.285)
(m-2-3.75) edge node [right] {$\cI_2$} (m-1-3.285)
(m-2-4.75) edge node [right] {$\cI_3$} (m-1-4.285)
(m-1-1.255) edge node [left] {$\cR_0$} (m-2-1.105)
(m-1-2.255) edge node [left] {$\cR_1$} (m-2-2.105)
(m-1-3.255) edge node [left] {$\cR_2$} (m-2-3.105)
(m-1-4.255) edge node [left] {$\cR_3$} (m-2-4.105)
    (m-1-1) edge node [above] {$\grad$} (m-1-2)
    (m-1-2) edge node [above] {$\curl$} (m-1-3)
    (m-1-3) edge node [above] {$\Div$}  (m-1-4)
    (m-2-1) edge node [above] {$\mathbb{G}$} (m-2-2)  
    (m-2-2) edge node [above] {$\mathbb{C}$} (m-2-3)
    (m-2-3) edge node [above] {$\mathbb{D}$}  (m-2-4)
    (m-4-1.105) edge node [left]  {$\tilde{\cR}_3$} (m-3-1.255)
    (m-4-2.105) edge node [left]  {$\tilde{\cR}_2$} (m-3-2.255)
    (m-4-3.105) edge node [left]  {$\tilde{\cR}_1$} (m-3-3.255)
    (m-4-4.105) edge node [left]  {$\tilde{\cR}_0$} (m-3-4.255)   
    (m-3-1.285) edge node [right] {$\tilde{\cI}_3$} (m-4-1.75)
(m-3-2.285) edge node [right] {$\tilde{\cI}_2$} (m-4-2.75)
(m-3-3.285) edge node [right] {$\tilde{\cI}_1$} (m-4-3.75)
(m-3-4.285) edge node [right] {$\tilde{\cI}_0$} (m-4-4.75)
    (m-4-4) edge node [above] {$-\grad$} (m-4-3)
    (m-4-3) edge node [above] {$\curl$} (m-4-2)
    (m-4-2) edge node [above] {$-\Div$}  (m-4-1)
    (m-3-4) edge node [above] {$-\tilde{\mathbb{G}}$} (m-3-3)  
    (m-3-3) edge node [above] {$\tilde{\mathbb{C}}$} (m-3-2)
    (m-3-2) edge node [above] {$-\tilde{\mathbb{D}}$}  (m-3-1)
    (m-3-1) edge node [left]  {$\hodgeDualThree$} (m-2-1)
    (m-3-2) edge node [left]  {$\hodgeDualTwo $} (m-2-2)
    (m-2-3) edge node [left]  {$\hodgeTwo $} (m-3-3)
    (m-2-4) edge node [left]  {$\hodgeThree$} (m-3-4)
    ;
\end{tikzpicture}
\end{equation}
In the diagram, we have denoted by $\matG$, $\matC$, $\matD$ the matrices representing the operators $\grad$, $\curl$, and $\Div$, respectively, on the level of the geometric degrees of freedom. Moreover, the discrete Hodge operators $\matH$ (and $\tilde{\matH}$ if it starts from the dual complex) are indexed corresponding to the starting space. We will choose to use in our formulation only 
$\matH_2$,  $\matH_3$, $\tilde{\matH}_2$,  $\tilde{\matH}_3$ and their inverses. Indeed, a choice needs to be made as \textit{e.g.} $\matH_0\neq \tilde{\matH}_3^{-1}$. 
Finally, $\cI_i$, $i=0,1,2,3$, denotes interpolation-histopolation operators that map the geometric degrees of freedom to a function in the continuous spaces. These are characterized by the right-inverse properties:
\begin{eqnarray}
\cR_{\ell}\left( \cI_{\ell} \cR_{\ell}(f)\right) = \cR_{\ell}(f) \text{   for all } f \in V_{\ell}, \quad \ell=0,1,2,3,  \label{eq:I_rinv}\\
\tilde{\cR}_{\ell}\left( \tilde{\cI}_{\ell} \tilde{\cR}_{\ell}(f)\right) = \tilde{\cR}_{\ell}(f) \text{   for all } f \in \tilde{V}_{\ell}, \quad \ell=0,1,2,3, \label{eq:tI_rinv}
\end{eqnarray}
where $V_{\ell}$ denotes the corresponding spaces $V_0 = H^1(\Omega)$, $V_1 = H(\curl,\Omega)$, $V_2 = H(\Div,\Omega)$ and $V_3 = L^2(\Omega)$. Operators with tilde denote the corresponding operators on the dual grids.

As remarked earlier and as will be seen in Section \ref{sec:conservation}, the discretized system will inherit many of the conservation laws when defining a commuting diagram. We call the diagram commuting if the following properties hold
\begin{itemize}
	\item The derivative operators form a commuting diagram with the reduction operators in the following sense:
	\begin{eqnarray}
	\cR_1(\grad \phi) &=& \matG \cR_0(\phi) \text{  for all } \phi \in H^1(\Omega);\label{eq:R1GR0} \\
	\cR_2(\curl \bF) &=& \matC \cR_1(\bF)  \text{  for all } \bF \in H(\curl, \Omega);\label{eq:R2CR1} \\
	\cR_3(\Div \bG) &=& \matD \cR_2(\bG)  \text{  for all } \bG \in H(\Div, \Omega); \label{eq:R3DR2}\\
	\tilde{\cR}_1(\grad \phi) &=& \tilde{\matG} \tilde{\cR}_0(\phi) \text{  for all } \phi \in H^1(\Omega); \label{eq:tR1GR0}\\
	\tilde{\cR}_2(\curl \bF) &=& \tilde{\matC} \tilde{\cR}_1(\bF)  \text{  for all } \bF \in H(\curl, \Omega); \label{eq:tR2CR1}\\
	\tilde{\cR}_3(\Div \bG) &=& \tilde{\matD} \tilde{\cR}_2(\bG)  \text{  for all } \bG \in H(\Div, \Omega). \label{eq:tR3DR2}
\end{eqnarray}
	\item Likewise the derivative operators form a commuting diagram with the inter\-po\-la\-tion-histopolation operators in the following sense:
		\begin{eqnarray}
		\grad(\cI_0 \arrphi) &=& \cI_1 \left( \matG \arrphi \right) \text{  for all } \arrphi \in \cC_0;\label{eq:I1G} \\		
		\curl(\cI_1 \arrF) &=& \cI_2 \left( \matC \arrF \right) \text{  for all } \arrF \in \cC_1;\label{eq:I2C} \\		
		\Div(\cI_2 \arrG) &=& \cI_3 \left( \matD \arrG \right) \text{  for all } \arrG \in \cC_2;\label{eq:I3D} \\		
		\grad(\tilde\cI_0 \arrphi) &=& \tilde\cI_1 \left( \tilde\matG \arrphi \right) \text{  for all } \arrphi \in \tilde\cC_0;\label{eq:I1Gd} \\		
		\curl(\tilde\cI_1 \arrF) &=& \tilde\cI_2 \left( \tilde\matC \arrF \right) \text{  for all } \arrF \in \tilde\cC_1;\label{eq:I2Cd} \\		
		\Div(\tilde\cI_2 \arrG) &=& \tilde\cI_3 \left( \tilde\matD \arrG \right) \text{  for all } \arrG \in \tilde\cC_2.\label{eq:I3Dd}
\end{eqnarray}
\end{itemize}
Finally, a natural choice for the discrete Hodge operators is
\begin{eqnarray}
\tilde{\matH}_3 = \cR_0 (\star \tilde{\cI}_3), \quad \tilde{\matH}_2 = \cR_1 (\star \tilde{\cI}_2), \quad \matH_2 = \tilde{\cR}_1 (\star \cI_2), \quad \matH_3 = \tilde{\cR}_0 (\star \cI_3),  	
\end{eqnarray}
where $\star$ denotes the classical continuous Hodge operator, but other choices are possible as we will see below.

\section{Tensor-product construction of mimetic finite differences}

After de\-fin\-ing this abstract framework, we will now exemplify reduction, interpolation and Hodge operators for which the diagram commutes on tensor product grids. 
As a construction principle in multiple dimensions we will rely on tensor products and construct the multidimensional operators from one dimensional operators. The dual grid setting proposed by Palha et al \cite{Palha:2014} is based on cells with Gauss--Lobatto points. The duality is directly introduced within each cell, the dual points being defined as Gauss points. This would be one option. 
We consider here translation invariant constructions based on Cartesian grids following the high order Finite Difference methods converging to pseudospectral methods as proposed in \cite{vincenti2016, fornberg1990}. A formulation of such mimetic finite differences for curvilinear grids has been provided by \cite{kreeft2011}. 
For an extension towards curvilinear grids along the lines of \cite{Perse2021} we refer to future work.

In the following, we provide a tensor product construction of reduction, interpolation and Hodge operators on a Cartesian domain $[0,L_x] \times [0,L_y] \times [0,L_z]$ with $M_{\alpha}$ nodes along each dimension $\alpha \in \{1,2,3\}$
\begin{equation} \label{gx}
\gx_\bsm = (x_{m_1}, y_{m_2}, z_{m_3})
  \qquad \text{ with } ~
  \bsm \in \range{1}{\bM} := \prod_{\alpha = 1}^3 \range{1}{M_\alpha}.
\end{equation}
The double brackets $\range{\cdot}{\cdot}$ denote the integer values in the given interval. 
Together with this primal grid, we consider the dual grid with points
\begin{equation} \label{gxd}
\gx_\bsm = (x_{m_1+\half}, y_{m_2+\half}, z_{m_3+\half})
  \qquad \text{ with } ~
  \bsm \in \range{1}{\bM} := \prod_{\alpha = 1}^3 \range{1}{M_\alpha}.
\end{equation}
For the position of the grid points on the dual grid, it holds that
\begin{equation*}
	x_{m_1 + \half} = \frac{x_{m_1}+x_{m_1+1}}{2}
\end{equation*}
and analogously for $y$ and $z$.

Since the construction principle will be tensor products of one-dimensional operators, let us also introduce the one-dimensional de Rham complex
\begin{equation*}\label{CGL1d}
\begin{tikzpicture}[baseline=(current  bounding  box.center)]
  \matrix (m) [matrix of math nodes,row sep=3em,column sep=4em,minimum width=2em] {
           H^1(\Omega) & L^2(\Omega)
    \\
           \mathcal{C}_0 &  \mathcal{C}_1 
    \\
           \tilde{\mathcal{C}}_1 &  \tilde{\mathcal{C}}_0
    \\
          L^2(\Omega) & H^1(\Omega)
    \\
    };
  \path[-stealth]
    (m-1-1) edge [transform canvas={xshift=-0.7mm}] node [left]  {$ \clr_0$} (m-2-1)    
    (m-2-1) edge [transform canvas={xshift=0.7mm}] node [right]  {$ \cli_0$} (m-1-1)
    (m-1-2) edge [transform canvas={xshift=-0.7mm}] node [left]  {$\clr_1$} (m-2-2)
    (m-2-2) edge [transform canvas={xshift=0.7mm}] node [right]  {$\cli_1$} (m-1-2)
    (m-1-1) edge node [above] {$\frac{\mathrm{d}}{\mathrm{d} x}$} (m-1-2)
    (m-2-1) edge node [above] {$\matd$} (m-2-2)  
    (m-4-2) edge [transform canvas={xshift=-0.7mm}] node [left]  {$\tilde{\clr}_0$} (m-3-2)
    (m-4-1) edge [transform canvas={xshift=-0.7mm}] node [left]  {$\tilde{\clr}_1$} (m-3-1) 
    (m-3-2) edge [transform canvas={xshift=0.7mm}] node [right]  {$\tilde{\cli}_0$} (m-4-2)
    (m-3-1) edge [transform canvas={xshift=0.7mm}] node [right]  {$\tilde{\cli}_1$} (m-4-1)
    (m-4-2) edge node [above] {$-\frac{\mathrm{d}}{\mathrm{d} x}$} (m-4-1)
    (m-3-2) edge node [above] {$\tilde{\matd}$} (m-3-1)  
    (m-3-1) edge [transform canvas={xshift=-0.7mm}]  node [left]  {$ \tilde{\matrh}_{1}$} (m-2-1)
    (m-2-1) edge [transform canvas={xshift=0.7mm}]  node [right]  {$ \matrh_{0}$} (m-3-1)
    (m-3-2) edge  [transform canvas={xshift=-0.7mm}]  node [left]  {$\tilde{\matrh}_{0}$} (m-2-2)
    (m-2-2) edge  [transform canvas={xshift=0.7mm}]  node [right]  {$\matrh_{1}$} (m-3-2)
    ;
\end{tikzpicture}
\end{equation*}	

\subsection{Reduction and derivative operators}\label{sec:reduction_operators}

For the definition of the reduction operators, we use the geometric degrees of freedom (cf.~\cite[Sec. 61.]{CPKS_variational_2020}).
\begin{definition}\label{def:restriction}
For functions $\phi \in H^1(\Omega)$, $\bE \in H(\curl,\Omega)$, $\bB \in H(\Div,\Omega)$ and $\rho \in L^2(\Omega)$ on a cartesian domain $\Omega$, let us define the restriction operators on the tensor product grid as follows:
\begin{itemize}
	\item $\cR_0: H^1(\Omega) \mapsto \cC_0$ which associates to a scalar function its values at all the vertices of the primal grid:
	$ \cR_{0,(i,j,k)}(\phi) = \phi(x_i,y_j,z_k) =: \arrphi_{i,j,k}.$
	\item $\cR_1:H(\curl,\Omega) \mapsto \cC_1$ which associates to a vector valued function the circulations on all the edges.  As there are three edges associated to a vertex we define
	$$\cR_{1,(i,j,k)} (\bE) = (\cR_{1,(i,j,k)} ^x(E_x), \cR_{1,(i,j,k)}^y (E_y), \cR_{1,(i,j,k)}^z (E_z))^\top,$$ with
	$\cR_{1,(i,j,k)}^x (E_x) = \int_{x_i}^{x_{i+1}} E_x(x,y_j,z_k) \dd x =: \arrE_{i+\half,j,k} =: \arrE^x_{i,j,k}$ the edge integral of $E_x$ along the $x$ direction, and similarly for the edges in the $y$ and $z$ directions. 
	\item $\cR_2:H(\Div,\Omega) \mapsto \cC_2$ associates to a vector valued function the fluxes through all faces.  As there are three faces associated to a vertex we define 
   $$\cR_{2,(i,j,k)} (\bB) = (\cR_{2,(i,j,k)}^x (B_x), \cR_{1,(i,j,k)}^y (B_y), \cR_{2,(i,j,k)}^z (B_z))^\top,$$ with
	$\cR_{2,(i,j,k)}^x (\bB) = \int_{y_j}^{y_{j+1}} \int_{z_{k}}^{z_{k+1}} B_x(x_i,y,z) \dd y \dd z =: \arrB_{i,j+\half,k+\half}=: \arrB^x_{i,j,k}$, the flux through the face orthogonal to the $x$ direction and similarly for the faces orthogonal to the $y$ and $z$ directions.
	\item $\cR_3: L^2(\Omega) \mapsto \cC_3$ which associates to a scalar function its integrals on all the cells of the grid: 
	$\cR_{3,(i,j,k)}(\rho) = \int_{x_i}^{x_{i+1}}\int_{y_j}^{y_{j+1}}\int_{z_k}^{z_{k+1}} \rho(x,y,z) \dd x  \dd y \dd z = \arrrho_{i,j,k}.$
\end{itemize}
The restriction operators for the dual sequence are defined analogously using the vertices, edges, faces and cells on the dual mesh:
\begin{itemize}
	\item $\tilde{\cR}_0: H^1(\Omega) \mapsto \tilde{\cC}_0$ which associates to a scalar function its values at all the vertices of the dual grid:
	$ \tilde{\cR}_{0,(i,j,k)}(\phi) = \phi(x_{i+\half},y_{j+\half},z_{k+\half}) = \tilde{\arrphi}_{i+\half,j+\half,k+\half}.$
	\item $\tilde{\cR}_1:H(\curl,\Omega) \mapsto \tilde{\cC}_1$ which associates to a vector valued function the circulations on all the edges of the dual grid. For the edges along the $x$ direction, we define
	$\tilde{\cR}_{1,(i,j,k)}^x (\bH) = \int_{x_{i-\half}}^{x_{i+\half}} H_x(x,y_{j+\half },z_{k+\half}) \dd x =: \tilde{\arrH}_{i,j+\half,k+\half} =: \tilde{\arrH}^x_{i,j,k}$, and similarly for the edges in the $y$ and $z$ directions. 
	\item $\tilde{\cR}_2:H(\Div,\Omega) \mapsto \tilde{\cC}_2$ which associates to a vector valued function the integrals on all the fluxes through all the faces of the dual grid. For the faces orthogonal to the $x$ direction, we define
	$$\tilde{\cR}_{2,(i,j,k)}^x (\bD) = \int_{y_{j-\half}}^{y_{j+\half}} \int_{z_{k-\half}}^{z_{k+\half}} D_x(x_{i+\half},y,z) \dd z \dd y =: \tilde{\arrD}_{i+\half,j,k} =: \arrD^x_{i,j,k},$$ and similarly for the faces orthogonal to the $y$ and $z$ directions.
	\item $\tilde{\cR}_3: L^2(\Omega) \mapsto \tilde{\cC}_3$ which associates to a scalar function its integrals on all the cells of the dual grid: 
	$\tilde{\cR}_{3,(i,j,k)}(\rho) = \int_{x_{i-\half}}^{x_{i+\half}}\int_{y_{j-\half}}^{y_{j+\half}}\int_{z_{k-\half}}^{z_{k+\half}} \rho(x,y,z) \dd z  \dd y \dd z = \tilde{\arrrho}_{i,j,k}.$
\end{itemize}
\end{definition}

We observe that with these notations for the reduction operators the indices of primal $p$-forms and dual $3-p$ forms perfectly match, and we can define a duality product
\begin{align*}
	\arrphi\cdot\tilde{\arrrho} &= \sum_{i,j,k} \arrphi_{i,j,k} \tilde{\arrrho}_{i,j,k}, \\
	\arrE\cdot\tilde{\arrD} &= \sum_{i,j,k} \arrE_{i+\half,j,k}\tilde{\arrD}_{i+\half,j,k}
	+ \arrE_{i,j+\half,k}\tilde{\arrD}_{i,j+\half,k} + \arrE_{i,j,k+\half}\tilde{\arrD}_{i,j,k+\half}\\
	&= \arrE^x\cdot\tilde{\arrD}^x + \arrE^y\cdot\tilde{\arrD}^y +\arrE^z\cdot\tilde{\arrD}^z\\
	\arrB\cdot\tilde{\arrH} &= \sum_{i,j,k} \arrB_{i,j+\half,k+\half}\tilde{\arrH}_{i,j+\half,k+\half}
	+ \arrB_{i+\half,j,k+\half}\tilde{\arrH}_{i+\half,j,k+\half} + \arrB_{i+\half,j+\half,k}\tilde{\arrH}_{i+\half,j+\half,k}\\
	&= \arrB^x\cdot\tilde{\arrH}^x + \arrB^y\cdot\tilde{\arrH}^y +\arrB^z\cdot\tilde{\arrH}^z \\
	\arrrho\cdot\tilde{\arrphi} &= \sum_{i,j,k} \arrrho_{i+\half,j+\half,k+\half} \tilde{\arrphi}_{i+\half,j+\half,k+\half}. 
\end{align*}  

Let us also define the following one-dimensional discrete derivative operator
 	\begin{equation*}
		 \mathbb{d} = \begin{pmatrix}
		-1 & 1 & 0 & \ldots & 0 \\
		0 & -1 & 1 & 0  & \\
		\vdots & & \ddots & \ddots & \\
		0 & & & -1 & 1 \\
		1 & 0 & \ldots & 0 & -1
		\end{pmatrix},
	\end{equation*} 
and its adjoint $\tilde{\matd} = - \matd^\top$. Then, we define the discrete $\grad$, $\Div$ and $\curl$ operators on the primal grid as blocks of Kronecker products of the one dimensional derivative and identity matrices: 
 \begin{equation} \label{hD0}
\matG = \begin{pmatrix}
\matI_{M_3} \otimes \matI_{M_2} \otimes \matd_{M_1} \\
\matI_{M_3} \otimes \matd_{M_2} \otimes \matI_{M_1} \\
\matd_{M_3} \otimes \matI_{M_2} \otimes \matI_{M_1}
\end{pmatrix},
\end{equation}
\begin{equation}\label{hD12}
  \matC = \begin{pmatrix}
  \matO_{M} & - \matd_{M_3} \otimes \matI_{M_2} \otimes \matI_{M_1}  & \matI_{M_3} \otimes \matd_{M_2} \otimes \matI_{M_3}\\
  \matd_{M_3} \otimes \matI_{M_2} \otimes \matI_{M_1}  & \matO_{M} &  -\matI_{M_3} \otimes \matI_{M_2} \otimes \matd_{M_1}\\
  -\matI_{M_3} \otimes \matd_{M_2} \otimes \matI_{M_1}  & \matI_{M_3} \otimes \matI_{M_2} \otimes \matd_{M_1} & \matO_{M} \\
  \end{pmatrix}
  \,  \text{ and } \,
  \matD =  \matG^\top.
\end{equation}
Their adjoint operators are given as:
 \begin{equation}\label{eq:relation_dual_derivative_ops}
 	\matG^\top = - \tilde{\matD}, \quad \matC^\top = \tilde{\matC}, \quad \matD^\top = - \tilde{\matG}.
 \end{equation}
 
\begin{theorem}
Let us consider a tensor product grid on a Cartesian domain and let the restriction operators be defined as in Definition~\ref{def:restriction} and the derivative operators as in \eqref{hD0}-\eqref{eq:relation_dual_derivative_ops}. Then, relations \eqref{eq:R1GR0}-\eqref{eq:tR3DR2} are satisfied.
\end{theorem}

\begin{proof}
	For $\phi \in H^1(\Omega)$, we have
	\begin{align*}
	\cR^x_{1,(i,j,k)}(\grad \phi) &=\int_{x_i}^{x_{i+1}} \frac{\dd\phi}{\dd x} (x,y_j,z_k)\dd x = \phi(x_{i+1},y_j,z_k) - \phi(x_{i},y_j,z_k) 
	\\
	&= (\matd_{M_1} \arrphi)_{i,j,k}.
	\end{align*}
	On the dual grid, we get
\begin{eqnarray*}
		&&\tilde{\cR}^x_{1,(i,j,k)}(\grad \phi) =\int_{x_{i-\half}}^{x_{i+\half}} \frac{\dd\phi}{\dd x} (x,y_{j+\half},z_{k+\half})\dd x = \phi(x_{i+\half},y_{j+\half},z_{k+\half}) \\
		&&\quad - \phi(x_{i-\half},y_{j+\half},z_{k+\half}) 
		=(-\matd_{M_1}^\top \arrphi)_{{i+\half},{j+\half},{k+\half}} = (\tilde{\matd}_{M_1} \arrphi)_{{i+\half},{j+\half},{k+\half}}.
\end{eqnarray*}
	We only consider the first component of the gradient, the other relations can be found analogously.
\end{proof}

\subsection{Interpolation-histopolation operators}

In order to evaluate a discrete function everywhere from its degrees of freedom in $\cC_i$ or $\tilde{\cC}_i$ $1\leq i \leq 3$, we need to define appropriate interpolation and histopolation operators (see \cite{robidoux2008}). Again the 3D operators can be constructed as products of one dimensional building blocks on a Cartesian mesh. 

Let us therefore construct the interpolation operator $\cli_0$ and the histopolation operator $\cli_1$ on the primal grid such that the following three properties are satisfied:
\begin{eqnarray}
\clr_0 \cli_0 \clr_0 (\phi) &=&\cli_0 \clr_0 (\phi) (x_i) = \clr_{0,i}(\phi)\text{ for }\phi \in V_0,\label{eq:prop_i0}\\ 
\clr_1 \cli_1 \clr_1(\psi) &=& \int_{x_i}^{x_{i+1}} (\cli_1 \clr_1(\psi))(x) \dd x = \clr_{1,i}(\psi)\text{ for }\psi \in V_1,\label{eq:prop_i1}\\
\frac{\dd}{\dd x} \left( \cli_0 \arrf \right) (x) &=& \left( \cli_1 \matd \arrf\right) (x)\text{ for }\arrf \in \cC_0.\label{eq:prop_i0i1}
\end{eqnarray}
These properties are the analogous properties to \eqref{eq:I_rinv} as well a \eqref{eq:I1G}-\eqref{eq:I3D} in three dimensions. Note that the restriction operator $\clr_0$ is again defined as point evaluation and $\clr_1$ as line integral.

A natural choice is to represent the functions interpolated to $V_0$ cell-wise by Lagrange polynomials $\ell_{i,\alpha}$ and the functions interpolated to $V_1$ by histopolation polynomials $\hbar_{i,\alpha}$ satisfying $\int_{x_{i+\gamma}}^{x_{i+\gamma+1}}\hbar_{i,\alpha}(x) \dd x = \delta_{\alpha,\gamma}$. Then, the first two properties are satisfied by definition. Indeed, a relation between these two types of functions can be established such that the third property is also satisfied. We summarize the definition in the following Theorem.

\begin{theorem}\label{thm:interpoaltion}
	Let $p \in \NN$ and define the $2p$ Lagrange polynomials $\ell_{i,\alpha}(x)$ satisfying $\ell_{i,\alpha}(x_j) = \delta_{i+\alpha,j}$, $\alpha= -p \ldots, p+1$. Moreover, define the $2p-1$ polynomials $\hbar_{i,\alpha}$, $\alpha =-p, \ldots, p$, by $\hbar_{i,\alpha} (x) = \sum_{\beta = \alpha+1}^{p+1} \ell'_{i,\beta}(x)$.  Then the interpolation operators defined on $[x_i,x_{i+1}]$ by
	\begin{eqnarray}
	\left( \cli_0 \arrf\right) (x) = \sum_{\alpha= -p}^{p+1} \arrf_{i+\alpha} \ell_{i,\alpha}(x), \text{ for } \arrf \in \cC_0, \label{def:i0}\\
	 \left( \cli_1 \arrg\right) (x) = \sum_{\alpha= -p}^p \arrg_{i+\alpha} \hbar_{i,\alpha}(x)  \text{ for } \arrg \in \cC_1, \label{def:i1}
	\end{eqnarray}
	satisfy the properties \eqref{eq:prop_i0}-\eqref{eq:prop_i0i1}.
\end{theorem}

\begin{proof}

Equation \eqref{eq:prop_i0} is satisfied by construction of the Lagrange polynomials. Next, we apply the $\clr_{1,\gamma}$ operator to $\hbar_{i,\alpha}$ for $\alpha,\gamma = -p,\ldots, p$:
\begin{eqnarray*}
	\clr_{1,i+\gamma}\left( \sum_{\beta = \alpha+1}^{p+1} \ell'_{i,\beta}(x)\right) &=& \sum_{\beta = \alpha+1}^{p+1} \left( \ell_{i,\beta}(x_{i+\gamma+1}) - \ell_{i,\beta} (x_{i+\gamma}) \right) = \sum_{\beta = \alpha+1}^{p+1}  \left( \delta_{\beta-1,\gamma} - \delta_{\beta,\gamma} \right) \\&=& \delta_{\alpha,\gamma},
\end{eqnarray*}
where we have used the nodal interpolation property of the Lagrange polynomials and resolved the telescoping sum in the last step and used the fact, that $\gamma \neq p+1$.\\
Finally, we show the last equality \eqref{eq:prop_i0i1}. It holds that
\begin{eqnarray*}
	\cli_1(\matd \arrf)(x) &=& \sum_{\alpha=-p}^p \left( \arrf_{i+\alpha+1} - \arrf_{i+\alpha} \right) \hbar_{i,\alpha}(x) = \sum_{\alpha=-p}^p \left( \arrf_{i+\alpha+1} - \arrf_{i+\alpha} \right)  \sum_{\beta = \alpha+1}^{p+1} \ell'_{i,\beta}(x) \\
	&=& - \arrf_{i-p} \sum_{\beta=-p+1}^{p+1} \ell'_{i,\beta}(x) + \sum_{\alpha= -p+1}^p \arrf_{i+\alpha} \ell'_{i,\alpha}(x) + \arrf_{i+p+1} \ell'_{i,p+1}(x) \\
	&=& \sum_{\alpha= -p}^{p+1} \arrf_{i+\alpha} \ell'_{i,\alpha}(x) = \frac{\dd}{\dd x} \cli_0 \arrf(x)
\end{eqnarray*}
where we have contracted the telescoping sum for the third equality and used that $\sum_{\beta=-p}^{p+1} \ell_{i,\beta}' = 0$ (and hence $-\sum_{\beta=-p+1}^{p+1} \ell_{i,\beta}' = \ell_{i,-p}'$) due to the partition of unity property of the Lagrange polynomials.
\end{proof}
Analogously, the operators $\tilde{\cli}_0$, $\tilde{\cli}_1$ on the dual grid can be defined. Based on these one dimensional building blocks, we can define the three-dimensional interpolation-histopolation operators: They are constructed by combining the Lagrange interpolation functions along directions where the restriction operation is based on point evaluation and the histopolation functions along directions where the restriction is based on integration. Based on Theorem \ref{thm:interpoaltion}, properties \eqref{eq:I1G}-\eqref{eq:I3D} as well as \eqref{eq:I_rinv}-\eqref{eq:tI_rinv} can be verified for these 3d interpolation-histopolation operators.

\subsection{Discrete Hodge operators}\label{sec:hodge_operators}

Based on the definition of the interpolation and restriction operators, we can now compute the discrete Hodge operators based on the interpolation-histopolation operators. 
On an equidistant tensor product grid, the lines of each Hodge operator are equivalent, and they can be built from a Kronecker product of 1d Hodge operators along each line. 

Let us therefore again focus on the 1d de Rham complex where we have the following Hodge operators 
\begin{equation*}
\tilde{\matrh}_1 = \clr_0 \star \tilde{\cli}_1, \quad {\matrh}_0 = \tilde{\clr}_1 \star \cli_0, \quad {\matrh}_1 = \tilde{\clr}_0 \star \cli_1, \quad \tilde{\matrh}_0 = \clr_1 \star \tilde{\cli}_0.
\end{equation*}
These Hodge operators are thus entirely defined through the definition of the restriction and interpolation-histopolation methods.

Note that the operator $\tilde{\matrh}_1$ transfers from an integral over the dual cells to the vertex value on the primal cell and ${\matrh}_1$ transfers from the integral over the primal cell to the vertex value on the dual grid. On the other hand, the operators ${\matrh}_0$ and $\tilde{\matrh}_0$ transfer from values to cell integrals between dual and primal grids.

Let us first consider the direction from point value to cell integrals and consider the point values $\arrf \in \cC_0$ on the primal grid. We consider the integral over the cell $[x_{i-\half},x_{i+\half}]$ of the dual grid
\begin{equation*}
\tilde{\clr}_{1,i}(\cli_0(\arrf)) = \int_{x_{i-\half}}^{x_{i+\half}} \cli_0(\arrf) \dd x.
\end{equation*}
In order to compute this, we need the representation of the interpolation operator on the cells $[x_{i-1},x_{i}]$ and $[x_{i},x_{i+1}]$ which are given according to the definition \eqref{def:i0} as
	\begin{eqnarray*}
	\left( \cli_0 \arrf\right) (x) = \begin{cases} \sum_{\alpha= -p-1}^{p} \arrf_{i+\alpha} \ell_{i-1,\alpha}(x) & x \in [x_{i-1},x_{i}] \\	
	 \sum_{\alpha= -p}^{p+1} \arrf_{i+\alpha} \ell_{i,\alpha}(x) & x \in [x_{i},x_{i+1}]
	\end{cases}
	\end{eqnarray*}
Then by integration over the two separate pieces, we can define the components of the Hodge operator
\begin{equation*}\begin{aligned}
	\tilde{\clr}_{1,i}\cli_0(\arrf) &=& \sum_{\alpha=-p}^p \arrf_{i+\alpha} \underbrace{\left( \int_{x_{i-\half}}^{x_{i}} \ell_{i-1,\alpha}(x) \dd x + \int_{x_i}^{x_{i+\half}} \ell_{i,\alpha}(x) \dd x \right)}_{=(\matrh_{0})_{i,i+\alpha}} \\
	 &+& \arrf_{i-p-1} \underbrace{\int_{x_{i-\half}}^{x_i} \ell_{i-1,\alpha}(x) \dd x}_{= (\matrh_{0})_{i,i-p-1}} +\arrf_{i-p-1} \underbrace{ \int_{x_{i}}^{x_{i+\half}} \ell_{i,\alpha}(x) \dd x}_{= (\matrh_0)_{i,i+p+1}}.
\end{aligned}\end{equation*}
Hence, the Hodge operator becomes a matrix with $2p+3$ diagonal elements.  Analogously, we can compute the entries of the Hodge operator $\tilde{\matrh}_0$ interchanging the roles of primal and dual points.

We note that this operation is translation invariant on an equidistant grid such that all lines are equal and $\matrh_{0/1} = \tilde{\matrh}_{0/1}$. The Hodge operator of orders 2,4,6 for the case of equidistant grids are given in the first part of Table \ref{sec:appendix_hodge}. As an example, we here compute the matrix for the case of $p=0$ in the definition of $\cli_0$. This yields the following representation
\begin{eqnarray*}
			\cli_0(\arrphi)(x) = \begin{cases} \arrphi_{i-1} \ell_{i-1,0}(x) + \arrphi_i \ell_{i-1,1} (x) & x_{i-\half}<x<x_i \\
			\arrphi_{i} \ell_{i,0}(x) + \arrphi_{i+1} \ell_{i,1}(x)  & x_{i}<x<x_{i+\half}. \end{cases}
			\end{eqnarray*}
With this expression, we can compute the restriction as 
$$\tilde{\clr}_{1,i} \cli_0(\arrphi) = \frac{h}{8} \left( \arrphi_{i-1} + 6 \arrphi_i + \arrphi_{i+1}\right)$$

Now we turn to the direction from integral to the point evaluation. We assume that the integral is defined on the primal grid, and we evaluate at the points of the dual grid. On the cell $[x_i,x_{i+1}]$, the histopolation $\cli_1$ is then defined for $\arrpsi \in \cC_1$ by \eqref{def:i1} as
\begin{equation*}
\left( \cli_1 \arrpsi\right) (x) = \sum_{\alpha= -p}^p \arrpsi^1_{i+\alpha} \hbar_{i,\alpha}(x) .
\end{equation*}
For the Hodge, we now evaluate this expression at $x_{i+\half}$ to find
\begin{equation*}
\left( \cli_1 \arrpsi\right) (x_{i+\half}) = \sum_{\alpha= -p}^p \arrpsi^1_{i+\alpha} \underbrace{\hbar_{i,\alpha}(x_{i+\half})}_{(\matrh_{1})_{i,i-p}} .
\end{equation*}
We conclude that the Hodge operator becomes a $2p+1$ diagonal matrix. In particular, for $p=0$, we get an identity matrix scaled by $\frac{1}{h}$ for the equidistant case.
 The corresponding operators for order 2, 4 and 6 can be found in the lower part of Table~\ref{sec:appendix_hodge}.
 
 \begin{table}\caption{Stencils of the Hodge operators for equidistant grid. The stencil is symmetric around the diagonal and cyclic for periodic boundary conditions.}\label{sec:appendix_hodge}
 \centering
 \begin{tabular}{|l|c | c|}
 \hline
 Hodge &order & stencil \\
 \hline
\multirow{3}{*}{$\matrh_0= \tilde{\matrh}_0$} & 2 & $\frac{h}{8} \quad \frac{6h}{8} \quad \frac{1h}{8}$ \\[1mm]
& 4 & $-\frac{7h}{384} \quad \frac{44h}{384} \quad \frac{310h}{384}\quad \frac{44h}{384}\quad -\frac{7h}{384} $\\[1mm]
& 6 & $\frac{163h}{46080} \quad -\frac{1114h}{46080} \quad \frac{4909h}{46080}\quad \frac{38164h}{46080}\quad \frac{4909h}{46080}\quad -\frac{1114h}{46080} \quad \frac{163h}{46080}$\\[1mm]
\hline
\multirow{3}{*}{$\matrh_0^{\min}= \tilde{\matrh}_0^{\min}$}& 2 & h \\[1mm]
& 4 & $\frac{h}{24} \quad \frac{22h}{24} \quad \frac{h}{24}$\\[1mm]
& 6 &  $-\frac{17h}{5760} \quad \frac{154h}{2880} \quad \frac{863h}{960}\quad \frac{154h}{2880}\quad -\frac{-17h}{5760} $\\[1mm]
 \hline
\multirow{3}{*}{$\matrh_1= \tilde{\matrh}_1$}& 2 & $\frac{1}{h}$ \\[1mm]
& 4 & $-\frac{1}{24 h} \quad \frac{13}{12h} \quad - \frac{1}{24h}$\\[1mm]
& 6 & $\frac{3}{640h} \quad -\frac{29}{480h} \quad \frac{1067}{960h}\quad -\frac{29}{480h}\quad \frac{3}{640h} $\\[1mm]
\hline
 \end{tabular}
 \end{table}

Note that for a fixed $p$, the stencil transferring from point to cell integrals has a width of $2p+3$ while the width of the stencil transferring from cell integrals to points has a width of $2p+1$ and the order of the Hodge operators is $2(p+1)$. On the other hand, it is possible to construct $\matrh_0^{ \min}$, $\tilde{\matrh}_0^{\min}$ operators of order $2(p+1)$ with a stencil of minimal width $2p+1$ by constructing a single Lagrange interpolation polynomial of the points $i-p, \ldots i+p$ and integrating it over the interval $[x_{i-\half},x_{i+\half}]$. This yields
$$
(\matrh_0^{\min} \arrf)_i = \sum_{\alpha=-p}^{p} \arrf_{i + \alpha} \int_{x_{i-\half}}^{x_{i+\half}} \ell_{i,\alpha}(x) \, \dd  x,
$$
where $\ell_{i,\alpha}$ denotes the Lagrange polynomial $\alpha$ for nodes $x_{i-p},\ldots, x_{i+p}$. For $p=0$ this yields, for example, the classical second order Yee scheme. The corresponding stencils of order 2, 4 and 6 can be found in the middle part of Table \ref{sec:appendix_hodge}. We note that these Hodge operators do not stem from interpolation-histopolation operators but
 the conservation properties that we will consider in Section \ref{sec:conservation} can all be shown by only requiring the Hodge operator to be symmetric. Let us thus define the following assumption that we base the rest of the paper on:
\begin{assumption}\label{assumption}
We assume that \eqref{eq:R1GR0}-\eqref{eq:tR3DR2} hold true, that $\matH_2$ and $\tilde{\matH}_2$ are symmetric positive definite matrices and that relation \eqref{eq:relation_dual_derivative_ops} holds between the derivative operators on the primal and dual grid.
\end{assumption}

\section{Semi-discrete Action principle}\label{sec:discrete_lagrangian}

\subsection{Derivation of the equations of motion}

In order to derive the equations of motion, we again follow \cite{CPKS_variational_2020} and introduce our discretizations into the Lagrangian \eqref{cL} and then derive the equations of motion as the Euler--Lagrange equations of the semi-discrete Lagrangian.

A general semi-discrete formulation of Low's Lagrangian based on the commuting diagram \eqref{CGL} was formulated in \cite{CPKS_variational_2020} as
\begin{equation}\begin{aligned} \label{cLh}
  &\cL_h (\bX_p, \dot{\bX}_p, \bV_p, \bA_h, \dot{\bA}_h, \phi_h)
  \\
  & \quad = \sum_{p=1}^N w_p \left( \big(m_s \bV_p + q_s \bA^S(\bX_p)\big)\cdot \dot{\bX}_p - \Big(\frac{m_s}{2}\bV^2_p +q_s\phi^S(\bX_p)\Big)\right)  \\
  &\quad +\frac {1}{2} \int_{\Omega} \varepsilon(x) |\grad\phi_h(\bx) +\dot{\bA}_h(\bx)|^2 \dd \bx  - \frac {1}{2}\int_{\Omega}  \frac{1}{\mu(x)}|\curl \bA_h(\bx)|^2  \dd \bx,
\end{aligned}\end{equation}
where $\bA^S$ and $\phi^S$ are obtained from the discrete fields $\bA_h$ and $\phi_h$ by a smoothing operator. Using the duality between the grids and the adequate reduction operators, we define the smoothed fields from their degrees of freedom on the grid and a smoothing kernel $S$ that will be taken to be a cardinal spline:
\begin{equation*}
	\bA^S(\bX_p) = \begin{pmatrix}
		\arrA^x \cdot \tilde{\cR}_{2}^x(S(\bx-\bX_p)) \\
		\arrA^y \cdot \tilde{\cR}_{2}^y(S(\bx-\bX_p)) \\
		\arrA^z \cdot \tilde{\cR}_{2}^z(S(\bx-\bX_p))
	\end{pmatrix}, \quad 
	\phi^S(\bx) = \arrphi \cdot \tilde{\cR}_{3}(S(\bx-\bX_p)). 
\end{equation*}
For conciseness, we will write in the sequel $S_p(\bx) := S(\bx-\bX_p)$.

For the discretization of the integrals based on the mimetic finite difference description, we can make a second order approximation based on the mean value theorem. Let us consider for example the integral 
$\int \bE(\bx) \cdot \bD(\bx) \dd \bx$
and define the reductions $\arrE = \cR_1(\bE) \in \cC_1$ and $\tilde{\arrD} = \tilde{\cR}_2(\bD) \in \tilde{\cC}_2$.

By the mean value theorem we get
\begin{equation*}
	\int_{x_i}^{x_{i+1}} \int_{y_{j-\half}}^{y_{j+\half}} \int_{z_{k-\half}}^{z_{k+\half}} \bE_x(\bx) \cdot \bD_x(\bx) \dd \bx = \arrE_{i+\half,j,k} \cdot \tilde{\arrD}_{i+\half,j,k} + \cO(\Delta x^2)
\end{equation*}
With this approach the semi-discrete Lagrangian can be represented as
\begin{eqnarray}\label{eq:Lh}
	\mathcal{L}_h &=& \sum_{p=1}^{N_p} w_p
	  \left( m \bV_p \cdot \dot{\bX}_p - \frac{m}{2} \bV_p^2 +  q \arrA \cdot \tilde{\cR}_2  \left( \dot{\bX}_p S_p(\bx)\right) - q\arrphi \cdot \tilde{\cR}_3(S_p(\bx)) \right)  \\
	  &&- \half \tilde{\arrD} \cdot \tilde{\matH}_2 \tilde{\arrD} - \half \left( \mathbb{C}\arrA \right) \cdot  \matH_2 \mathbb{C}\arrA +\tilde{\arrD} \cdot \left( -\dot{\arrA} - \mathbb{G} \arrphi \right)
	  \end{eqnarray}
where $\arrA\in\cC_1$ and $\arrphi\in\cC_3$ represent the arrays of degrees for freedom of $\bA_h$ and $\phi_h$.	  
The underlying definition of the electric and magnetic field degrees of freedom are
	  \begin{equation}\label{eq:eb_def}
	\arrE = - \dot{\arrA} - \matG \arrphi, \quad \arrB = \matC \arrA.
\end{equation}
Note that we have introduced the Hodge operator in order to only use one set of discrete variables, either on the primal (for the magnetic field) or on the dual grid (for the electric field). In the equations of motion, we will work with the degrees of freedom of the fields $\tilde{\arrD}$ and $\arrB$. 
 Also, we have added the term $\tilde{\arrD} \cdot \left( -\dot{\arrA} - \mathbb{G} \arrphi \right)$, since we want to work with the degrees of freedom of the electric field on the dual grid.

\begin{remark}
	Although we have motivated the coupling between primal and dual grid by a second order approximation of the $L^2$ scalar product, the dot product between quantities on the primal and the dual grid in the Lagrangian \eqref{eq:Lh} now defines a discrete duality and a high order approximation can be obtained from this Lagrangian by constructing high order Hodge operators.
\end{remark}

Now we can derive the equations of motion from the Euler--Lagrange equations. 
Let us first consider the variation with respect to $\arrA$. We have
\begin{equation*}
\fracp{\cL_h}{\arrA} = \sum_{p=1}^{N_p} w_p  q  \tilde{\cR}_{2}  \left( \dot{\bX}_p S_p(\bx)\right) - \mathbb{C}^\top  \matH_2 \mathbb{C}\arrA, \quad \fracp{\cL_h}{\dot{\arrA}} = -\tilde{\arrD}.
\end{equation*}
Using $\arrB = \matC \arrA \in \cC_2$, this yields the discrete Amp\`ere's law
\begin{equation*}
\frac{\dd \tilde{\arrD}}{\dd t}  = \mathbb{C}^\top  \matH_2 \arrB - \sum_{p=1}^{N_p} w_p  q  \tilde{\cR}_2  \left( \dot{\bX}_p S_p(\bx)\right).
\end{equation*}

Next let us consider the variation with respect to the scalar potential $\arrphi$
\begin{equation*}
	\fracp{ \cL_h}{ \arrphi} = -\sum_{p=1}^{N_p} w_pq  \tilde{\cR}_{3}(S_p(\bx))   +\tilde{\matD} \tilde{\arrD}, \quad \fracp{\cL_h}{\dot{\arrphi}} = 0, 
\end{equation*}
as $\matG^\top= - \tilde{\matD}$. This yields the electric Gauss' law
\begin{equation}\label{eq:disEGauss}
\tilde{\matD} \tilde{\arrD} = \sum_{p=1}^{N_p} w_pq  \tilde{\cR}_{3}(S(\bx- \bX_p)) .
\end{equation}

Taking the variation with respect to $\tilde{\arrD}$ yields
\begin{equation}\label{eq:def_etilde}
	-\tilde{\matH}_2 \tilde{\arrD} - \dot{\arrA} - \matG \arrphi = 0,
\end{equation}
since $\cL_h$ is independent of $\dot{\tilde{\arrD}}$. 
This equation yields Faraday's law when the discrete curl operator is applied, since $\arrB = \matC \arrA$ and $\matC \mat G = 0$ due to the complex property:
\begin{equation*}
	\frac{\dd \arrB}{\dd t} = -\matC \tilde{\matH}_2 \tilde{\arrD}.
\end{equation*}  
Finally, the magnetic Gauss law can be deduced from the definition of the magnetic field in \eqref{eq:eb_def} and the complex property $\matD \matC = 0$:
\begin{equation*}
	\matD \arrB = \matD \matC \arrA = 0.
\end{equation*}

Next, we turn to the particle equations of motion. The variation with respect to $\bV_p$ yields
\begin{equation}\label{eq:varVp}
	\frac{\dd\bX_p}{\dd t}=\bV_p.
\end{equation}

Finally, we consider the  variation with respect to $\bX_p$. Let us denote the unit vectors by $\uvec_\alpha$, $\alpha \in \{x,y,z\}$. Then we get
\begin{eqnarray*}
	\fracp{ \cL_h}{X_{p,\alpha}} &=& -w_p  q \arrA \cdot \tilde{\cR}_2  \left( \dot{\bX}_p \partial_\alpha S_p(\bx)\right) + q\arrphi\cdot \tilde{\cR}_3(\partial_\alpha S(\bx - \bX_p)) ,\\ 
	\fracp{\cL_h}{\dot{X}_{p,\alpha}} &=&  w_p m V_{p,\alpha} +  w_p  q \arrA \cdot \tilde{\cR}_2  \left( \uvec_\alpha S_p(\bx)\right).
\end{eqnarray*}
Hence,
\begin{equation*}
	\frac{\dd}{\dd t}\fracp{\cL_h}{\dot{X}_{p,\alpha}} =
	w_p \left(
  m \frac{V_{p,\alpha}}{\dd t} + q \dot{\arrA} \cdot \tilde{\cR}_2  \left( \uvec_\alpha S_p(\bx)\right)
	 - q \arrA \cdot \tilde{\cR}_2  ( \uvec_\alpha \dot{\bX}_p\cdot\nabla S_p(\bx))
	 \right).
\end{equation*}
As $\dot{\bX}_p=\bV_p$, $\uvec_\alpha \bV_p\cdot \nabla S-\bV_p \partial_\alpha S =\curl (\uvec_\alpha\times\bV_p S)$, $\partial_\alpha S=\Div (\uvec_\alpha S)$ and using \eqref{eq:varVp} this yields the following Euler--Lagrange equation
\begin{equation*}\label{eq:deriv_Lx0}
m \frac{\dd V_{p,\alpha}}{\dd t} = -  q \dot{\arrA} \cdot \tilde{\cR}_2  \left( \uvec_\alpha S_p(\bx)\right) + q\arrphi\cdot  \tilde{\cR}_3(\Div (\uvec_\alpha S_p(\bx))) +  q \arrA \cdot \tilde{\cR}_2  (\curl (\uvec_\alpha\times\bV_p S_p(\bx))) .
\end{equation*}
This becomes, using properties \eqref{eq:tR2CR1} and \eqref{eq:tR3DR2}
\begin{equation}\label{eq:deriv_Lx}
	m \frac{\dd V_{p,\alpha}}{\dd t} = -  q \dot{\arrA} \cdot \tilde{\cR}_2  \left( \uvec_\alpha S_p(\bx)\right) + q\arrphi\cdot  \tilde{\matD}\tilde{\cR}_2(\uvec_\alpha S_p(\bx)) +  q \arrA \cdot \tilde{\matC}\tilde{\cR}_1   (\uvec_\alpha\times\bV_p S_p(\bx))  .
\end{equation}
Now have that $\tilde{\matD}^\top = -\matG$ and $\tilde{\matC}^\top = \matC$, and we can use the definition of the fields from the potentials \eqref{eq:eb_def}.  
Therefore, equation \eqref{eq:deriv_Lx} becomes
\begin{equation*}\label{eq:deriv_Lx2}
	\begin{aligned}
	\frac{\dd \bV_{p}}{\dd t} &=& \frac{q}{m} \left(  \bE^S(\bX_p) + \bV_p \times \bB^S(\bX_p) \right).
\end{aligned}
\end{equation*}
Using $\arrE = \tilde{\matH}_2 \tilde{\arrD}$ we have
\begin{equation*}
	\bE^S(\bX_p) := \sum_{\alpha}  \Big((\tilde{\matH}_2 \tilde{\arrD}) \cdot \tilde{\cR}_{2}(\uvec_\alpha S_p(\bx)) \Big) \uvec_{\alpha}, \;
	\bB^S(\bX_p) := \sum_{\alpha} \Big(\arrB \cdot\tilde{\cR}_1( \uvec_\alpha S_p(\bx))\Big) \uvec_{\alpha} ,
\end{equation*}
or equivalently writing explicitly the three components
\begin{equation*}
	\bE^S(\bX_p) = \begin{pmatrix}
		(\tilde{\matH}_2 \tilde{\arrD})^x \cdot \tilde{\cR}_{2}^x(S_p(\bx)) \\
		(\tilde{\matH}_2 \tilde{\arrD})^y \cdot \tilde{\cR}_{2}^y(S_p(\bx)) \\
		(\tilde{\matH}_2 \tilde{\arrD})^z \cdot \tilde{\cR}_{2}^z(S_p(\bx))
	\end{pmatrix},
\quad
	\bB^S(\bX_p) = \begin{pmatrix}
		\arrB^x \cdot \tilde{\cR}_{1}^x(S_p(\bx)) \\
		\arrB^y \cdot \tilde{\cR}_{1}^y(S_p(\bx)) \\
		\arrB^z \cdot \tilde{\cR}_{1}^z(S_p(\bx))
	\end{pmatrix}.
\end{equation*}

\subsection{Discrete noncanonical Hamiltonian system}

The discrete equations of motion can be collected as
\begin{eqnarray}
\frac{\dd \bX_p}{\dd t} &=& \bV_p,\label{eq:eqns_of_motion_part_1} \\
 \frac{\dd V_{p}^x}{\dd t} &=& \frac{q}{m} \left( 
	(\tilde{\matH}_2 \tilde{\arrD})^x \cdot \tilde{\cR}_{2}^x(S_p(\bx)) + V_{p}^y\tilde{\arrB}^z \cdot \tilde{\cR}_{1}^z(S_p(\bx)) -  V_{p}^z\tilde{\arrB}^y \cdot \tilde{\cR}_{1}^y(S_p(\bx))
 \right)\label{eq:eqns_of_motion_part_2x} \\
 \frac{\dd V_{p}^y}{\dd t} &=& \frac{q}{m} \left( 
	(\tilde{\matH}_2 \tilde{\arrD})^y \cdot \tilde{\cR}_{2}^y(S_p(\bx)) + V_{p}^z\tilde{\arrB}^x \cdot \tilde{\cR}_{1}^x(S_p(\bx)) -  V_{p}^x\tilde{\arrB}^z \cdot \tilde{\cR}_{1}^z(S_p(\bx))
 \right)\label{eq:eqns_of_motion_part_2y} \\
 \frac{\dd V_{p,z}}{\dd t} &=& \frac{q}{m} \left( 
	(\tilde{\matH}_2 \tilde{\arrD})^z \cdot \tilde{\cR}_{2}^z(S_p(\bx)) + V_{p}^x\tilde{\arrB}^y \cdot \tilde{\cR}_{1}^y(S_p(\bx)) -  V_{p}^y\tilde{\arrB}^x \cdot \tilde{\cR}_{1}^x(S_p(\bx))
 \right)\label{eq:eqns_of_motion_part_2z} \\
 	\frac{\dd\arrB }{\dd t} &=& -\matC \tilde{\matH}_2 \tilde{\arrD},\label{eq:eqns_of_motion_part_3} \\
 	\frac{\dd \tilde{\arrD}}{\dd t}  &=& \mathbb{C}^\top  \matH_2 \arrB - \sum_{p=1}^{N_p} w_p  q  \tilde{\cR}_2  \left( \bV_p S(x-\bX_p)\right) \label{eq:eqns_of_motion_part_4}.
\end{eqnarray}
In addition to these dynamical equations, we have the divergence constraints
\begin{eqnarray}\label{eq:div_constraints}
\tilde{\matD} \tilde{\arrD} = \sum_{p=1}^{N_p} w_pq  \tilde{\cR}_{3}(S_p(\bx)), \quad\matD \arrB = 0.
\end{eqnarray}
Note that the second equation follows from the definition $\arrB = \matC \arrA$ and the property $\matD \matC = 0$.

Let us introduce the matrix $\matS^2(\arrX) \in (\RR^{3\times 3})^{N_p \times N_2}$ with diagonal $3\times 3$ blocks
\begin{equation*} \label{matS}
  \matS^2(\arrX)_{p,i} = 
  \diag(\tilde{\cR}_{2,i}^x(S_p(\bx)),\tilde{\cR}_{2,i}^y(S_p(\bx)),\tilde{\cR}_{2,i}^z(S_p(\bx)) )
  \, \text{ for } \, 1 \le p \le N_p, \, 1 \le i \le N_2,
\end{equation*}
grouping the indices $(i,j,k)$ into one unique index $i$.
 Moreover, we introduce  $\matr(\bb) = \big((\uvec_\alpha \times \uvec_\beta ) \cdot \bb \big)_{1 \le \alpha, \beta \le 3}
\in \RR^{3\times 3}$ to be the rotation matrix
\begin{equation*}\label{matr}
\matr(\bb) = \begin{pmatrix}
   0 & b_3 & - b_2\\
   -b_3 & 0 & b_1 \\
   b_2 & -b_1 & 0
  \end{pmatrix}
  \quad \text{ such that }
  \quad \bv \times \bb  = \matr(\bb) \bv \quad \text{ for all } \bv, \bb \in \RR^3,
\end{equation*}
and we denote by
$\matR^1(\arrX, \arrB) \in (\RR^{3\times 3})^{N_p \times N_p}$
the block-diagonal rotation matrix with blocks
\begin{equation}\begin{aligned}\label{eq:matR1}
	\matR^1(\arrX, \arrB)_{p,p} &=  \matr\left(\sum_{i=1}^{N_1}\tilde{\cR}_{1,i}(\arrB_i S_p(\bx)\right)\\&=\matr((\tilde{\cR}_{1}^x(S_p(\bx))\cdot\arrB^x,\tilde{\cR}_{1}^y(S_p(\bx))\cdot\arrB^y,\tilde{\cR}_{1}^z(S_p(\bx))\cdot \arrB^z ).
\end{aligned}\end{equation}
Finally, we collect the weights $w_p$ times $q$, $m$, or $q/m$ on the diagonals of the matrices $\matW_q, \matW_m, \matW_{q/m} \in \RR^{3N_p \times 3N_p}$, respectively, and denote by $\matW_{1/m}=\matW_m^{-1}$.

With this notation, we can rewrite the system \eqref{eq:eqns_of_motion_part_1}-\eqref{eq:eqns_of_motion_part_4} as
\begin{eqnarray*}
\frac{\dd \arrX}{\dd t} &=& \arrV, \\
 \frac{\dd \arrV}{\dd t} &=& \matW_{\frac{q}{m}} \left( \matS^2(\arrX) \tilde{\matH}_2\tilde{\arrD} + \matR^1(\arrX, \arrB) \arrV  \right) \\
 	\frac{\dd\arrB}{\dd t}  &=& -\matC \tilde{\matH}_2 \tilde{\arrD}, \\
 	\frac{\dd \tilde{\arrD}}{\dd t}  &=& \mathbb{C}^\top  \matH_2 \arrB - \matW_{q} \matS^2(\arrX) \arrV.
\end{eqnarray*}

Let us discretize the Hamiltonian analogously to the Lagrangian as
\begin{equation}\label{eq:ham}
\cH(\arrU) 
= \tfrac 12 \arrV^\top \matW_m \arrV + \tfrac 12 \tilde{\arrD}^\top  \tilde{\matH}_2 \tilde{\arrD} + \tfrac 12 \arrB^\top \matH_2 \arrB,
\end{equation}
Collecting the dynamical variables into one vector $\arrU^\top = ( \arrX^\top, \, \arrV^\top, \,  \tilde{\arrD}^\top, \, \arrB^\top)$
we can represent the equations of motion as the following non-canonical Hamiltonian system
\begin{equation} \label{ham}
  \frac{\dd \arrU}{\dd t} = \matJ(\arrU)\nabla_\arrU \cH(\arrU) 
\end{equation}
with a structure matrix given by 
\begin{equation}\label{eq:poisson-mat}
 \matJ(\arrU) =
 \begin{pmatrix}
   0 &  \matW_{\frac 1m} & 0& 0 \\
  -\matW_{\frac 1m} & \matW_{\frac qm} \matR^1(\arrX, \arrB) \matW_{\frac 1m} & \matW_{\frac qm} \matS^2(\arrX) & 0\\
  0 & -\matS^2(\arrX)^\top\matW_{\frac qm} & 0 & \matC \\
  0& 0 &   \matC^\top &0\\
  \end{pmatrix},
\end{equation}
which is an antisymmetric matrix of dimension $3N_p+3N_p+N_1+N_2.$

\subsection{Semi-discrete conservation properties} \label{sec:conservation}

In this section, we show the conservation properties of the semi-discrete system under the Assumption \ref{assumption}.

\begin{theorem}
The matrix defined in \eqref{eq:poisson-mat} is a Poisson matrix, i.e. it is antisymmetric and satisfies the Jacobi identity for all $i,j,k \in \{1,\ldots, 6N_p + N_1 + N_2\}$
\begin{equation}\label{eq:jacobi_identity}
\sum_\ell \left( \frac{\partial \matJ_{i,j}(\arrU)}{\partial \arrU_\ell}\matJ_{\ell,k} + \frac{\partial \matJ_{jk}(\arrU)}{\partial \arrU_\ell}\matJ_{\ell,i} + \frac{\partial \matJ_{k,i}(\arrU)}{\partial \arrU_\ell} \matJ_{\ell,j}  \right) = 0.
\end{equation}
\end{theorem}

\begin{proof}
We note that the blocks corresponding to $\bX$ and $\bV$ have $3 N_p$ lines or columns, the blocks corresponding to $\tilde{\arrD}$ have $N_1$  lines or columns, and the blocks corresponding to $\arrB$ have $N_2$ lines or columns.
	Also, we note that $\matJ$ only depends on $\arrX$ and $\arrB$. Hence, we only need to consider the lines corresponding to $\arrX$ and $\arrB$ in the sum on $\ell$, \textit{i.e.} $\ell \in \{1, \ldots, 3N_p\}$ and $\ell \in \{6N_p+N_1+1, \ldots, 6N_p+N_1+N_2\}$. 
	In order to identify the terms that do not trivially vanish, let us consider the block
	form of the matrix, denoting by $\matJ_{X,V}$ the block with lines corresponding to $\arrX$ and columns to $\arrV$. Not counting the antisymmetric counterparts only 
	$\matJ_{X,V}$, $\matJ_{V,V}$, $\matJ_{V,D}$, and $\matJ_{D,B}$ do not vanish. On the other hand for the derivatives, only the blocks $\fracp{\matJ_{V,V}}{\arrX}$,  $\fracp{\matJ_{V,D}}{\arrX}$
	and $\fracp{\matJ_{V,V}}{\arrB}$ do not vanish. So in the sums involved in  \eqref{eq:jacobi_identity}, only products corresponding to terms in 
	$$\fracp{\matJ_{V,V}}{\arrX}\matJ_{X,V}, \;\fracp{\matJ_{V,V}}{\arrB}\matJ_{D,B}, \;  
	\fracp{\matJ_{V,D}}{\arrX}\matJ_{D,B}$$ 
	do not vanish. So we have only two cases to consider:
	\begin{enumerate}
		\item All indices $i,j,k \in V=\{3N_p+1, \ldots, 6N_p\}$
	    \item Two indices, $i,j\in  V=\{3N_p+1, \ldots, 6N_p\}$ and one index $k\in D = \{6N_p+1,\ldots, 6N_p+N_1\}$.
	\end{enumerate}
	Let us first consider the case that $i,j,k \in V=\{3N_p+1, \ldots, 6N_p\}$. Then the sum of $\ell$ is also restricted to $\{3N_p+1, \ldots, 6N_p\}$. Let us define $\hat{i} = i - 3N_p$,  $\hat{j} = j - 3N_p$,  $\hat{k} = k - 3N_p$, so that we get the following non-trivial line of the Jacobi identity
	\begin{multline*}
		\sum_{\ell=1 }^{3N_p} \left( \frac{\partial \left(\matW_{\frac qm} \matR^1(\arrX, \arrB) \matW_{\frac 1m}\right)_{\hi,\hj}}{\partial \arrX_\ell} \left( \matW_{\frac{1}{m}}\right)_{\ell,\hk} + \frac{\partial \left(\matW_{\frac qm} \matR^1(\arrX, \arrB) \matW_{\frac 1m}\right)_{\hj,\hk}}{\partial \arrX_\ell} \left( \matW_{\frac{1}{m}}\right)_{\ell,\hi} 
		\right. \\ \left.
		+ \frac{\partial \left(\matW_{\frac qm} \matR^1(\arrX, \arrB) \matW_{\frac 1m}\right)_{\hi,\hj}}{\partial \arrX_\ell} \left( \matW_{\frac{1}{m}}\right)_{\ell,\hk} \right).
	\end{multline*}
	Let us note that the $\matW$ matrices are diagonal, hence only one value of $\ell$ contributes with a nonzero contribution in each term, and we can split the terms $\left(\matW_{\frac qm} \matR^1(\arrX, \arrB) \matW_{\frac 1m}\right)_{\cdot, \cdot}$. This yields
		\begin{equation*}\begin{aligned}
		\left(\matW_{\frac qm}\right)_{\hi,\hi} \frac{\partial \left( \matR^1(\arrX, \arrB) \right)_{\hi,\hj}}{\partial \arrX_{\hk}}\left(\matW_{\frac 1m}\right)_{\hj,\hj} \left( \matW_{\frac{1}{m}}\right)_{\hk,\hk} \\
		+\left(\matW_{\frac qm}\right)_{\hj,\hj} \frac{\partial \left( \matR^1(\arrX, \arrB) \right)_{\hj,\hk}}{\partial \arrX_{\hi}}\left(\matW_{\frac 1m}\right)_{\hk,\hk} \left( \matW_{\frac{1}{m}}\right)_{\hi,\hi}\\
		+\left(\matW_{\frac qm}\right)_{\hk,\hk} \frac{\partial \left( \matR^1(\arrX, \arrB) \right)_{\hk,\hi}}{\partial \arrX_{\hj}}\left(\matW_{\frac 1m}\right)_{\hi,\hi} \left( \matW_{\frac{1}{m}}\right)_{\hj,\hj}.
	\end{aligned}\end{equation*}
	Since $\matR^1(\arrX,\arrB)$ is a block-diagonal matrix with blocks belonging to the three indices of each particle and only depending on the $\arrX$ coordinates belonging to this particle, we only have a nonzero value if $\hi,\hj,\hk$ all belong to the same particle. In this case, we can disregard all the weight terms, since they are identical for all indices of one particle (in case of a zero weight, the term trivially gives zero). Let us consider the multi-indices $(p,\alpha)$, $(p,\beta)$, and $(p,\gamma)$ instead of $\hi,\hj,\hk$ to specify the particle number $p$ and the component by $\alpha, \beta,\gamma \in \{1,2,3\}$. We thus have to consider the following expression
	\begin{multline*}
		\frac{\partial \left( \matR^1(\arrX, \arrB) \right)_{(p,\alpha),(p,\beta)}}{\partial \arrX_{p,\gamma}} + \frac{\partial \left( \matR^1(\arrX, \arrB) \right)_{(p,\beta),(p,\gamma)}}{\partial \arrX_{p,\alpha}} + \frac{\partial \left( \matR^1(\arrX, \arrB) \right)_{(p,\gamma),(p,\alpha)}}{\partial \arrX_{p,\beta}} = \\
-\sum_{i=1}^{N_1}\left( \tilde{\cR}_{1,i}\left( \left(\uvec_\alpha \times \uvec_\beta \right) \cdot \arrB_i\partial_\gamma S_p(\bx) \right) + \tilde{\cR}_{1,i}\left( \left(\uvec_\beta \times \uvec_\gamma \right) \cdot \arrB_i\partial_\alpha S_p(\bx) \right) \right. \\ \left. 
+ \tilde{\cR}_{1,i}\left( \left(\uvec_\gamma \times \uvec_\alpha \right)\cdot \arrB_i \partial_\beta S_p(\bx) \right)  \right)	
	\end{multline*}
	where we have used the definition \eqref{eq:matR1} for the $\matR^1$ term. If two components are the same, this term vanishes by the definition of the cross product. Therefore, all terms are zero, if $\alpha= \beta = \gamma$. If two components are equal and one is different, one of the term is zero and the other two cancel out each other due to the antisymmetry of the cross product. Finally, if all terms are different, we have the three terms of the gradient, i.e.
		\begin{equation*}
\left( \arrB^x\cdot\tilde{\cR}_1^x\left(  \partial_x S_p(\bx) \right) + \arrB^y\cdot\tilde{\cR}_1^y\left(  \partial_y S_p(\bx) \right) + \arrB^z\cdot\tilde{\cR}_1^z\left(  \partial_z S_p(\bx) \right)  \right)		=
 \arrB \cdot\tilde{\cR}_1(\grad S_p(\bx))   .
	\end{equation*}
Now, we use the relation between the reduction operators and the gradient to find:
			\begin{equation*}
- \arrB\cdot\tilde{\cR}_1(\grad S_p(\bx))  = -\arrB\cdot\left(\tilde{\matG} \tilde{\cR}_0(S_p(\bx))\right)  
=  \matD\arrB\cdot\tilde{\cR}_0(S_p(\bx)) = 0,
	\end{equation*}
where we have use property \eqref{eq:relation_dual_derivative_ops} and the fact that $\matD=-\tilde{\matG}$.

	Now, we turn to the second case involving the blocks $(V,V)$, $(V,D)$ and $(D,V)$ \ie $i,j \in \{3N_p+1, \ldots, 6N_p\}$ and $k \in \{6N_p+1,\ldots, 6N_p+N_1\}$. Let us again define the shifted indices $\hat{i} = i - 3N_p$,  $\hat{j} = j - 3N_p$,  $\hat{k} = k - 6N_p$. Then, the non-trivially vanishing line of the Jacobi identity reads
	\begin{multline}\label{eq:proof_jacobi_1}
	\sum_{\ell=1 }^{N_2}  \frac{\partial \left(\matW_{\frac qm} \matR^1(\arrX, \arrB) \matW_{\frac 1m}\right)_{\hi,\hj}}{\partial \arrB_\ell} \matC_{\ell,\hk} \\ + \sum_{\ell=1 }^{3N_p} \left(\frac{\partial \left(\matW_{\frac qm} \matS^2(\arrX) \right)_{\hj,\hk}}{\partial \arrX_\ell} \left( -\matW_{\frac{1}{m}}\right)_{\ell,\hi} + \frac{\partial \left(- \matS^2(\arrX)^\top \matW_{\frac qm}\right)_{\hk,\hi}}{\partial \arrX_\ell} \left( -\matW_{\frac{1}{m}}\right)_{\ell,\hk} \right).
	\end{multline}
	As the matrix $\matW_{\frac 1m}$ is diagonal, we only have a non-zero term if $\ell=\hi$ in the second and $\ell=\hk$ in the third term.

	The block diagonal structure of $\matR^1$ also implies that the first term vanishes if $\hi$ and $\hj$ do not belong to the same particle, say $p$, and moreover $\hi \neq \hj$, since the diagonal of the block vanishes. If $\hi$ and $\hj$, do not belong to the same particle, also the second and third term vanish, since $\matS^2(\arrX)_{\hj,\hk}$ only depends on $\arrX_{\hi}$ if $\hi$ belongs to the same particle as $\hj$ (and the same with $\hj$ and $\hi$ interchanged). We also note that the terms are the same with opposite sign if $\hj=\hi$. Finally, let us consider the case where $\hi$ is expressed as $(p,\alpha)$ and $\hj$ as $(p,\beta)$ with $\alpha,\beta \in \{1,2,3\}$ and $\alpha \neq \beta$. Then, \eqref{eq:proof_jacobi_1} can be written as
	\begin{multline*}
		\sum_{\ell=1 }^{N_2}  \frac{\partial \left( \matR^1(\arrX, \arrB) \right)_{(p,\alpha),(p,\beta)}}{\partial \arrB_\ell} \matC_{\ell,\hk} - \frac{\partial \left( \matS^2(\arrX) \right)_{(p,\beta),\hk}}{\partial \arrX_{p,\alpha}} + 
		\frac{\partial \left(\matS^2(\arrX)\right)_{(p,\alpha),\hk}}{\partial \arrX_{(p,\beta)}}\\
		= \sum_{\ell=1 }^{N_2}  \frac{\partial \left( \matR^1(\arrX, \arrB) \right)_{(p,\alpha),(p,\beta)}}{\partial \arrB_\ell} \matC_{\ell,\hk} - \frac{\partial \left( \matS^2(\arrX) \right)_{(p,\beta),\hk}}{\partial \arrX_{p,\alpha}} + 
		\frac{\partial \left(\matS^2(\arrX)\right)_{(p,\alpha),\hk}}{\partial \arrX_{(p,\beta)}}.
	\end{multline*}
	Note that we have left out the weight terms, since the particle weight is the same for all components of the same particle, so that we have one constant multiplying the whole line.
	Using the definition \eqref{eq:matR1} for the first term, we get
	\begin{eqnarray*}
		&&\sum_{\ell=1 }^{N_2}  \frac{\partial \left( \matR^1(\arrX, \arrB) \right)_{(p,\alpha),(p,\beta)}}{\partial \arrB_\ell} \matC_{\ell,\hk} = \sum_{\ell=1}^{N_2} \tilde{\cR}_{1,\ell} \left( \uvec_{\alpha} \times \uvec_{\beta} S_p(\bx)\right)^\top \matC_{\ell,\hk}\\
		 && \quad= \sum_{\ell=1}^{N_2}  \tilde{\matC}_{\hk,\ell}\tilde{\cR}_{1,\ell} \left( \uvec_{\alpha} \times \uvec_{\beta} S_p(\bx)\right) = \tilde{\cR}_{2,\hk} \left( \curl \left( \uvec_\alpha \times \uvec_\beta S_p(\bx) \right) \right),
	\end{eqnarray*}
	where we have used the property \eqref{eq:relation_dual_derivative_ops} for the second and property \eqref{eq:tR2CR1} for the last equality. Next, we use that $\matS^2(\arrX)_{(p,\alpha),\hk} = \tilde{\cR}_{2,\hk} \left( \uvec_\alpha S_p(\bx) \right)$. Therefore, we can transform the second and third term to
	\begin{eqnarray*}
		 - \frac{\partial \left( \matS^2(\arrX) \right)_{(p,\beta),\hk}}{\partial \arrX_{p,\alpha}} + 
		\frac{\partial \left(\matS^2(\arrX)\right)_{(p,\alpha),\hk}}{\partial \arrX_{(p,\beta)}} &=& \tilde{\cR}_{2,\hk} \left( \uvec_\beta \partial_\alpha S_p(\bx) \right) - \tilde{\cR}_{2,\hk} \left( \uvec_\alpha \partial_\beta S_p(\bx) \right) \\
		&=& \tilde{\cR}_{2,\hk} \left( \curl \left(  \uvec_\beta \times \uvec_\alpha S_p(\bx) \right) \right).
	\end{eqnarray*}
	Since $\uvec_\beta \times \uvec_\alpha = - \uvec_\beta \times \uvec_\alpha$, the sum of the three terms is zero.
\end{proof}

We note that the Poisson matrix is independent on the Hodge operator so that this property holds true also when using Hodge operators of the same stencil width.

\begin{lemma}
	System \eqref{ham} with Poisson matrix \eqref{eq:poisson-mat} and Hamiltonian \eqref{eq:ham} conserves energy.
\end{lemma}

\begin{proof}
	This follows immediately from the form \eqref{ham} and the fact that \eqref{eq:poisson-mat} is antisymmetric.
\end{proof}

\begin{lemma}
	The system \eqref{ham} conserves the divergence constraints \eqref{eq:div_constraints}, if they are satisfied initially.
\end{lemma}
\begin{proof}
Applying $\matD$ to \eqref{eq:eqns_of_motion_part_3} yields
\begin{equation*}
	\matD\frac{\dd\arrB}{\dd t}  = -\matD \matC \tilde{\matH}_2 \tilde{\arrD}.
\end{equation*}
Since $\matD \matC = 0$ and the time derivative and the matrix $\matD$ converge, we get
$\frac{\dd}{\dd t} \matD \arrB $.
Next, let us apply $\tilde{\matD}$ to \eqref{eq:eqns_of_motion_part_4}
\begin{equation}\label{lem_div_constraints:1}
\tilde{\matD}\frac{\dd \tilde{\arrD}}{\dd t}  = \tilde{\matD}\mathbb{C}^\top  \matH_2 \arrB -\tilde{\matD} \sum_{p=1}^{N_p} w_p  q  \tilde{\cR}_2  \left( \bV_p S(x-\bX_p)\right)
\end{equation}
First we note that $\matC^\top = \tilde{\matC}$ and hence it holds that $\tilde{\matD}\matC^\top = 0$ and the first term on the right-hand-side vanishes. Furthermore, taking the time derivative of $S(\bx-\bX_p)$, we get the discrete continuity equation
\begin{equation*}
	\frac{\dd}{\dd t} S(\bx-\bX_p) = - \bV_p\cdot\nabla S(\bx-\bX_p) = -\Div (\bV_p S(\bx-\bX_p)). 
\end{equation*}
Applying $\tilde{\cR}_3$ to this equation and using property \eqref{eq:tR3DR2}, we get
\begin{equation*}
-\tilde{\matD} \sum_{p=1}^{N_p} w_p  q  \tilde{\cR}_2  \left( \bV_p S(x-\bX_p)\right) = \frac{\dd }{\dd t} \sum_{p=1}^{N_p} w_pq  \tilde{\cR}_{3}(S(\bx - \bX_p)).
\end{equation*} 
Hence, equation \eqref{lem_div_constraints:1} corresponds to the total time derivative of the discrete Gauss law \eqref{eq:div_constraints}, which shows that this discrete Gauss law is satisfied at all times provided it is satisfied at the initial time.
\end{proof}
Note that the properties of the last two lemmas only depend on the properties $\matD \matC = 0$ and $\tilde{\matD}\tilde{\matC} = 0$, and, hence, not on how we have constructed the Hodge operator.

\section{Propagation in time}

Until now, we have considered a semi-discrete system, and we still need to discretize in time in order to be able to perform numerical simulations. The semi-discrete system has the same form as the semi-discrete system derived from a finite element exterior calculus discretization of the Vlasov--Maxwell system. For this reason, the very same time propagation schemes can be employed also in this case. In particular, one can use a Hamiltonian splitting to obtain an explicit scheme (cf.~\cite{he2015hamiltonian,kraus2016gempic}) or we can use the fact that the system is a skew-symmetric gradient system and derive energy-conserving (semi-)implicit discrete gradient schemes as in \cite{kormann2020energy}. Moreover, subcycling algorithms as discussed in \cite{hirvijoki2020, kormann2020energy} can also be derived. In our numerical experiments, we will consider the explicit time propagation scheme based on the Hamiltonian splitting as in \cite{kraus2016gempic}. This integrator is a Poisson integrator, i.e. the divergence constraints remain satisfied over time, but there is an oscillatory error in energy that is of the order $\cO(\Delta t^q)$, where $q$ is the order of the splitting scheme applied, i.e. $q=2$ for a Strang splitting, for instance.

\section{Summary of the full scheme}

\begin{enumerate}
\item Initialization:
\begin{enumerate}
\item Sample the particles from the initial distribution: $\arrX^{(0)}$, $\arrV^{(0)}$.
\item Accumulate the charge density $\arrrho^{(0)} \in {\cC}_3$ according to the right-hand side of \eqref{eq:disEGauss} and solve the Poisson equation $- \tilde{\matD} \tilde{\matH}_2^{-1} \matG \arrphi^{(0)} =\arrrho^{(0)}$ for $\arrphi^{(0)} \in \tilde{\cC}_0$ and initialize the self-consistent electric field as $\tilde{\arrD}^{(0)} = \tilde{\matH}_2 \matG \arrphi^{(0)}$.\\
Note that the Poisson equation can be obtained by combining \eqref{eq:disEGauss} and \eqref{eq:def_etilde} with the gauge condition $\Div \bA = 0$.
\item Add background magnetic field by projecting the given function.
\end{enumerate}
\item Time loop $n-1 \rightarrow n$ (Strang splitting)
\begin{enumerate}
\item Half time step of source-free Maxwell's equation
 $$	\arrB^*  = \arrB^{n-1} - \frac{\Delta t}{2}\matC \tilde{\matH}_2 \tilde{\arrD}^{n-1}, \quad
 	\tilde{\arrD}^*  = \mathbb{C}^\top  \matH_2 \arrB^{n-1}$$
\item Set $\tilde{\arrJ}^{(n)}=0$ and loop over particles and update particle $p$ as (denoting by $\varepsilon_{\alpha,\beta,\gamma}$ the Levi--Civita symbol)
\begin{enumerate}
\item $\arrV_p^* = \arrV_{p}^{(n-1)} + \frac{\Delta t}{2} \frac{w_p q_p}{m_p}  \matS^2(\arrX_p^{(n-1)}) \tilde{\matH}_2\tilde{\arrD}^*$
\item $\arrX_{p,x}^{\diamond}(\tau) = \arrX_{p,x}^{(n-1)}+ \tau \arrV_{p,x}^*$, $\arrX_{p,x}^* = \arrX_{p,x}^{\diamond}(\frac{\Delta t }{2})$, \\
$\arrV^{**}_{p,\alpha} = \arrV^*_{p,\alpha} + \varepsilon_{\alpha,x,\beta} \frac{w_p q_p}{m_p} \int_0^{\Delta t/2}\matR_{\beta}^1(\arrX_p^{\diamond}(\tau), \arrB^*) \arrV_{p,x}^* \dd \tau, \, \alpha,\beta \in \{y,z\},$\\
$\tilde{\arrJ}^{(n)}_{x} = \tilde{\arrJ}^{(n)}_{x}+  w_p q_p \int_0^{\Delta t/2} \matS_x^2(\arrX_p^{\diamond}(\tau)) \arrV^*_{p,x} \dd \tau$\\
\item $\arrX_{p,y}^{\diamond}(\tau) = \arrX_{p,y}^{(n-1)}+ \tau \arrV_{p,y}^{**}$, $\arrX_{p,y}^* = \arrX_{p,y}^{\diamond}(\frac{\Delta t }{2})$, \\
$\arrV^{***}_{p,\alpha} = \arrV^{**}_{p,\alpha} +\varepsilon_{\alpha,y,\beta} \frac{w_p q_p}{m_p} \int_0^{\Delta t/2}\matR_{\beta}^1(\arrX_p^{\diamond}(\tau), \arrB^*) \arrV_{p,y}^{**} \dd \tau,\, \alpha,\beta \in \{x,z\}$ \\
$\tilde{\arrJ}^{(n)}_{y} = \tilde{\arrJ}^{(n)}_{y}+ w_p q_p \int_0^{\Delta t/2} \matS_y^2(\arrX_p^{\diamond}(\tau)) \arrV^{**}_{p,y} \dd \tau$
\item $\arrX_{p,z}^{\diamond}(\tau) = \arrX_{p,z}^{(n-1)}+ \tau \arrV_{p,z}^{***}$, $\arrX_{p,z}^* = \arrX_{p,z}^{\diamond}(\Delta t )$, \\
$\arrV^{****}_{p,\alpha} = \arrV^{***}_{p, \alpha} +\varepsilon_{\alpha,z,\beta} \frac{w_p q_p}{m_p} \int_0^{\Delta t}\matR_{\beta}^1(\arrX_p^{\diamond}(\tau), \arrB^*) \arrV_{p,z}^{***} \dd \tau, \, \alpha,\beta\in\{x,y\}$ \\
$\tilde{\arrJ}^{(n)}_{z} = \tilde{\arrJ}^{(n)}_{z} +w_p q_p \int_0^{\Delta t} \matS_z^2(\arrX_p^{\diamond}(\tau)) \arrV^{***}_{p,z} \dd \tau$
\item Repeat ii., repeat i. (always updating the current values of $\arrX$, $\arrV$ $\rightarrow \arrX^{(n)}, \arrV^{+}$
\end{enumerate}
\item Full time step of source term in Ampere's law
$$
\tilde{\arrD}^{**} = \tilde{\arrD}^{**} + \tilde{\arrJ}^{(n)}
$$
\item Loop over particles and update particle $p$ as 
$$\arrV_p^{(n)} = \arrV_{p}^{+} + \frac{\Delta t}{2} \frac{w_p q_p}{m_p}  \matS^2(\arrX_p^{(n)}) \tilde{\matH}_2\tilde{\arrD}^{**}$$
\item Half time step of source-free Maxwell's equation
 $$	\arrB^n  = \arrB^* - \frac{\Delta t}{2}\matC \tilde{\matH}_2 \tilde{\arrD}^{**}, \quad
 	\tilde{\arrD}^n  = \mathbb{C}^\top  \matH_2 \arrB^*$$
\end{enumerate}
Note that $\arrX^{\diamond}$ has the time-dependent component as defined in the substep and uses the latest values of $\arrX^*$ along the other two directions.
\end{enumerate}

\section{Numerical experiments}

We have implemented the finite difference framework in our novel GEMPICX code that is based on the AMReX framework \cite{AMReX_JOSS}. Building on this framework we can leverage domain decomposition, performance portability and in future works adaptive mesh refinement and embedded boundary treatment of complex geometries.

\subsection{Hodge solver}

As a first step, we verify the order of the Hodge solvers $\matH_2$ and $\tilde{\matH}_2$ that determine the final order of the scheme as the other parts are exact. For this purpose, we consider the vector field $\bF(x,y,z) =	(\cos(x+y+z) \, -2\cos(x+y+z) \, \cos(x,y,z))^\top$ 
on the domain $[0,4\pi]^3$. We compute ${\arrF}_i = {\cR}_i(\bF)$, $\tilde{\arrF}_i = \tilde{\cR}_i(\bF)$, $i=1,2$. Note that we use Gauss--Legendre quadrature of the corresponding order to approximate the integrals in the restriction operation as also in all other numerical experiments, except for the Yee scheme where we increase the order, since the Hodge operator and the approximate restriction of the same order would be the same. Then, we apply the Hodges and compute the errors $e_1 = \|\arrF_1- \tilde{\matH}_2 \tilde{\arrF}_2\|_2$ and $e_2 = \|\tilde{\arrF}_1- {\matH}_2 \arrF_2\|$ for various grids and Hodge operator approximations. We consider the interpolation-histopolation Hodges of order 2, 4, 6, 8, 10 and 12 as well as the minimal stencil Hodges where we replace the point-to-integral Hodges with the minimal stencil-width versions.  Table \ref{tab:hodge} shows the absolute errors for various order of the Hodge as well as the resulting convergence orders. The order of the solver is recovered in all cases, noting that we hit round-off errors for orders 10 and 12 on the finest grid.

\begin{table}\caption{Absolute errors $e_1 = \|\arrF_1- \tilde{\matH}_2 \tilde{\arrF}_2\|_2$ and $e_2 = \|\tilde{\arrF}_1- {\matH}_2 \arrF_2\|$ on grids with $n$ cells per direction together with numerically observed convergence order (comparing the grid on the same line with the previous one) for Hodge operators of various orders. The first six Hodges are the natural ones while the operators $\matrh_0,\tilde{\matrh}_0$ are changed for the minimal width ones for the last two. Note that round-off errors are hit on finest grid for Hodges of order 10 and 12 (the corresponding entries have been striked out in the table).}\label{tab:hodge}
\begin{center}
\begin{tabular}{|c|rr|rr|rr|rr|}
\hline
 & \multicolumn{4}{c}{order 2} & \multicolumn{4}{|c|}{order 4} \\
n & $e_1$ & c.o. & $e_2$ & c.o. & $e_1$ & c.o. & $e_2$ & c.o.\\
\hline
16 &214. &      & 323.5 &  & 25.16 &  & 27.22 & \\
32 & 58.1 & 1.89 & 59.3 & 1.97  & 1.76 & 3.83 & 1.80 & 3.92\\
64 & 14.8 & 1.97 & 14.9 & 1.99 & $1.13 \cdot 10^{-1}$ & 3.96 & $1.14 \cdot 10^{-1}$ & 3.98\\
128 & 3.72  & 1.99 & 3.74 & 2.00 & $7.14 \cdot 10^{-3}$ &  4.00 & $7.14 \cdot 10^{-3}$ &  4.00\\
\hline
\hline
& \multicolumn{4}{c}{order 6} & \multicolumn{4}{|c|}{order 8} \\n & $e_1$ & c.o. & $e_2$ & c.o. & $e_1$ & c.o. & $e_2$ & c.o.\\
\hline
16 &  2.71 &  &  2.92 &  & $3.23 \cdot 10^{-1}$ & & $3.49 \cdot 10^{-1}$ & \\
32 &  $4.92 \cdot 10^{-2}$ & 5.78 & $5.01 \cdot 10^{-2}$ & 5.87 & $1.52 \cdot 10^{-3}$ & 7.73 &  $1.55 \cdot 10^{-3}$ & 7.81\\
64 & $7.97 \cdot 10^{-4}$ & 5.95 & $8.01 \cdot 10^{-4}$ & 5.97 & $6.23 \cdot 10^{-6}$ & 7.93  & $6.26 \cdot 10^{-6}$ & 7.95\\
128 & $1.26 \cdot 10^{-5}$ & 5.99 & $1.26 \cdot 10^{-5}$ & 5.99 & $2.46 \cdot 10^{-8}$ & 7.98 & $2.46 \cdot 10^{-8}$ & 7.99\\
\hline
\hline
& \multicolumn{4}{c}{order 10} & \multicolumn{4}{|c|}{order 12} \\n & $e_1$ & c.o. & $e_2$ & c.o. & $e_1$ & c.o. & $e_2$ & c.o.\\
\hline
16 & $4.05 \cdot 10^{-2}$ &  &  $4.38 \cdot 10^{-2}$ &  & $5.24 \cdot 10^{-3}$ & & $5.67 \cdot 10^{-3}$ & \\
32 &  $4.96 \cdot 10^{-5}$ & 9.67 & $5.06 \cdot 10^{-5}$ & 9.76 & $1.67 \cdot 10^{-6}$ & 11.61 &  $1.70 \cdot 10^{-6}$ & 11.70\\
64 & $5.13 \cdot 10^{-8}$ & 9.92 & $5.15 \cdot 10^{-8}$ & 9.94 & $4.10 \cdot 10^{-10}$ & 11.99  & $4.11 \cdot 10^{-10}$ & 12.01\\
128 & $7.58 \cdot 10^{-11}$ & 9.40 & \sout{$7.57 \cdot 10^{-11}$} & \sout{5.99} & \sout{$2.87 \cdot 10^{-11}$} & \sout{3.83} & \sout{$2.88 \cdot 10^{-11}$} & \sout{3.84}\\
\hline
\hline
 & \multicolumn{4}{c}{order 2 (Yee)} & \multicolumn{4}{|c|}{order 4 (minimal stencil)} \\
n & $e_1$ & c.o. & $e_2$ & c.o. & $e_1$ & c.o. & $e_2$ & c.o.\\
\hline
16 &73.2 &      & 79.2 &  & 6.87 &  & 7.44 & \\
32 & 19.5 & 1.91 & 19.9 & 2.00  & $4.74 \cdot 10^{-1}$ & 3.86 & $4.83 \cdot 10^{-1}$ & 3.94\\
64 & 4.95 & 1.98 & 4.97 & 2.00 & $3.04 \cdot 10^{-2}$ & 3.96 & $3.05 \cdot 10^{-2}$ & 3.99\\
128 & 1.24  & 1.99 & 1.24 & 2.00 & $1.91 \cdot 10^{-3}$ &  3.99 & $1.91 \cdot 10^{-3}$ &  4.00\\
\hline
\end{tabular}
\end{center}
\end{table}

\subsection{Electromagnetic waves in Vlasov--Maxwell system}

The advantage of a higher order Maxwell solver can be seen when looking at the numerical dispersion relation. Let us consider a two species simulation with reduced mass ratio $m_i/m_e = 10$ and $q_i = - q_e$. The initial conditions are Gaussians of the form
\begin{equation*}
f_s(\bx,\bv,t=0) = \frac{1}{\pi^{3/2} v_{th,s}^3}\exp \left(-\half\frac{|\bv|^2}{v_{th,s}} \right),
\end{equation*}
where $v_{th,e} = 0.05c$ and $v_{th,i} = \frac{0.05c}{\sqrt{10}}$. 
With a constant exterior field of strength $B_0$ along the $z$-axis, two electromagnetic waves, the so-called L- and R-modes, are triggered by the particle noise which have the respective dispersion relations of the form (see \cite[Sec. 3.4]{KilianDis})
\begin{equation}\label{eq:dispersion_lr}
	k^2 = \frac{\omega^2}{c^2} \left( 1 - \frac{\omega_{p}^2}{(\omega + \omega_{c,e})(\omega - \omega_{c,i})}\right), \quad 
k^2 = \frac{\omega^2}{c^2} \left( 1 - \frac{\omega_{p}^2}{(\omega - \omega_{c,e})(\omega + \omega_{c,i})}\right),
\end{equation}
where $\omega_{c,s} = \frac{q_s B_0}{m_s c}$.

We consider a quasi one-dimensional simulation with a domain of size $[0,64d_e] \times [0,d_e] \times [0,d_e]$, where $d_e$ is the electron inertial length, over a time of $T=200\omega_{p,e}^{-1}$. We use $128 \times 6 \times 6$ grid points for the field solvers and 400 particles per cell for both species generated by quasi-random numbers. The smoothing kernel is given by splines of degree $p=2$ and the time step is $\Delta t = 0.05\frac{1}{\omega_{pe}}$. 
Figure \ref{fig:lrmode} shows the numerical dispersion relation along the $x_1$ axis of $E_2$ (averaged over $x_2,x_3$) for the simulation with a Maxwell solver of order 2, 4, 6 and 12 together with the analytic dispersion relation \eqref{eq:dispersion_lr} for a quarter of the frequency space. It can be seen that the numerical dispersion is in good agreement with the analytic one for low wave number while it becomes worse with increasing wave number. For medium to high wave numbers, the accuracy of the numerical dispersion relation, with respect to the analytical one, improves as the order of the Maxwell solver increases.

\begin{figure}[ht]
\begin{subfigure}{.49\textwidth}
  \centering
  \includegraphics[width=\linewidth]{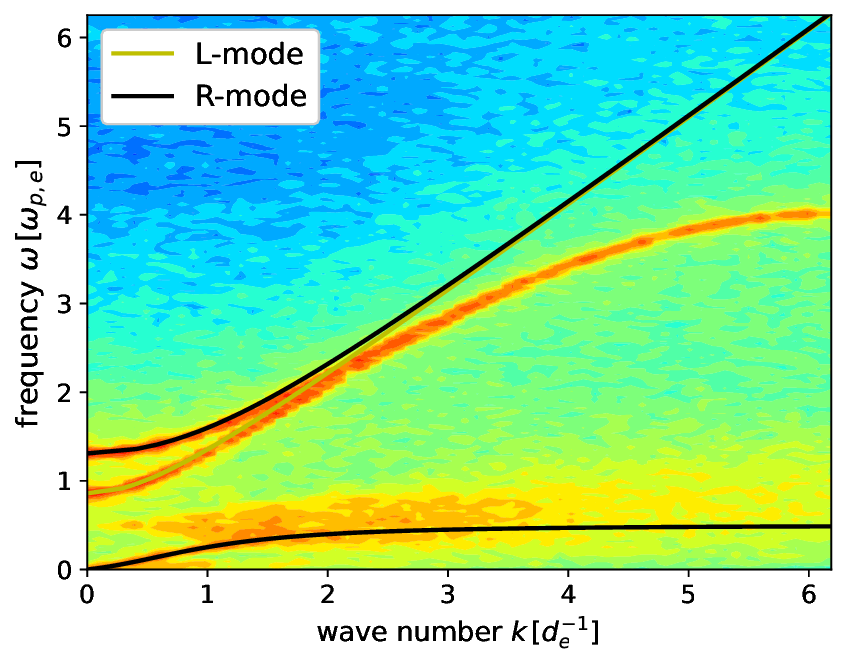}  
  \caption{Maxwell solver of order 2}
  \label{fig:nd_2}
\end{subfigure}
\begin{subfigure}{.49\textwidth}
  \centering
  \includegraphics[width=\linewidth]{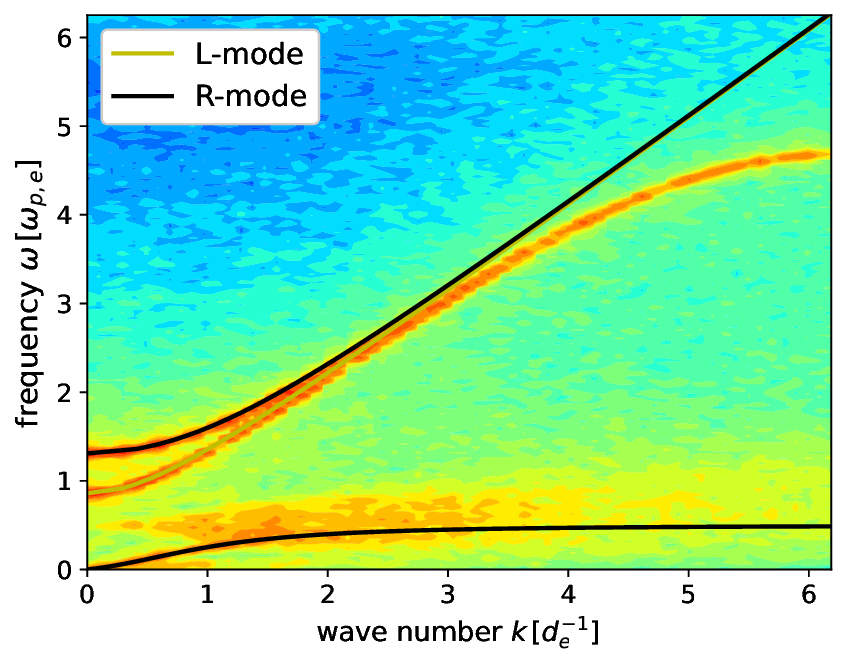}  
  \caption{Maxwell solver of order 4}
  \label{fig:nd_4}
\end{subfigure}
\begin{subfigure}{.49\textwidth}
  \centering
  \includegraphics[width=\linewidth]{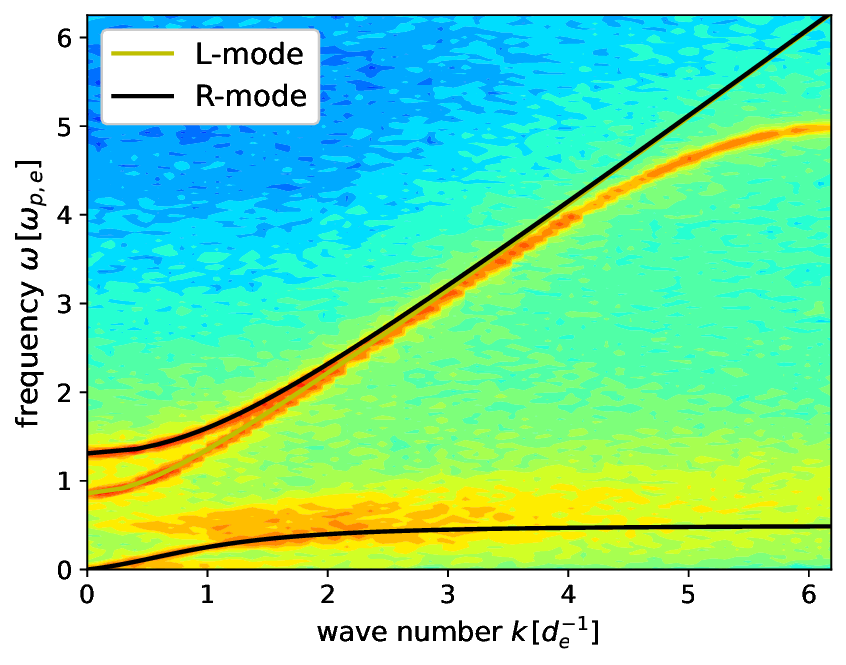}  
  \caption{Maxwell solver of order 6}
  \label{fig:nd_6}
\end{subfigure}
\begin{subfigure}{.49\textwidth}
  \centering
  \includegraphics[width=\linewidth]{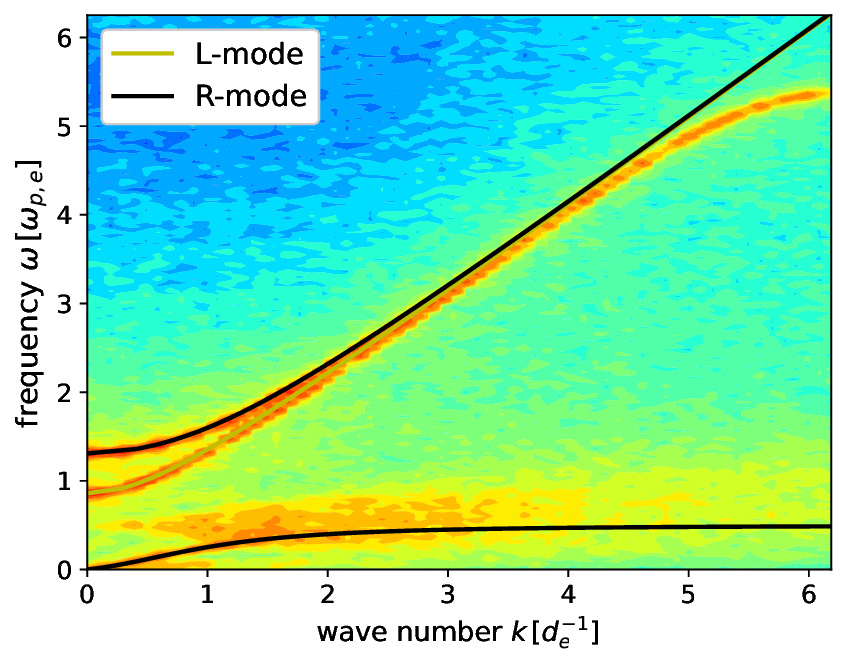}  
  \caption{Maxwell solver of order 12}
  \label{fig:nd_12}
\end{subfigure}
\caption{Comparison of the numerical dispersion relation for the L- and R-mode for varying order of the Maxwell solver (interpolation-histopolation Hodges).}
\label{fig:lrmode}
\end{figure}

\subsection{Weibel instablity in Vlasov--Maxwell system}

We now consider the Vlasov--Maxwell system in 3d3v phase-space with the following initial value of the distribution function for the electron distribution (in a neutralizing ion background)
\begin{equation*}
f(\bx,\bv,t=0) = \frac{1}{\pi \sigma_1 \sigma_2^2} \exp \left( - \half \left( \frac{v_1^2}{\sigma_1^2} + \frac{v_2^2 + v_3^2}{\sigma_2^2} \right) \right)
\end{equation*}
and vanishing initial magnetic field on the domain $[0, \frac{2\pi}{k}]^3$. This is known as the Weibel instability test case. We choose the following parameters: $\sigma_1 = \frac{0.02}{\sqrt{2}}$, $\sigma_2 = \sqrt{12}\sigma_1$, $k=1.25$, and $\beta = 10^{-4}$.

\subsubsection{Conservation properties}

As a first step, we numerically verify the conservation properties of our code with simulations on a $7 \times 7 \times 7$ points, 2000 particles per cell. The fully discrete algorithm conserves Gauss' law. In Figure \ref{fig:gauss}, we show the time evolution of $\|\nabla \bD(\bx,t) - \rho(\bx,t)\|_2/\|\rho(\bx,t)\|_2$. We can see that Gauss' law is indeed conserved to round off precision. In these experiments, we have used a time step of $\Delta t = 0.02$.

Next we consider conservation of energy. Energy is conserved at the semi-discrete level but, since we are choosing a Poisson integrator, energy conservation is lost on the fully discrete level. 
Figure \ref{fig:energy} shows the maximum energy error over all time steps on the time interval $[0,500]$ as a function of the temporal step size $\Delta t$ for a number of choices for the Maxwell solver and degree of the shape functions. We see that in all cases---in particular also with the minimal width stencil---the second order of the used Strang splitting scheme is recovered.

\begin{figure}[ht]
\begin{subfigure}{.49\textwidth}
  \centering
  \includegraphics[width=\linewidth]{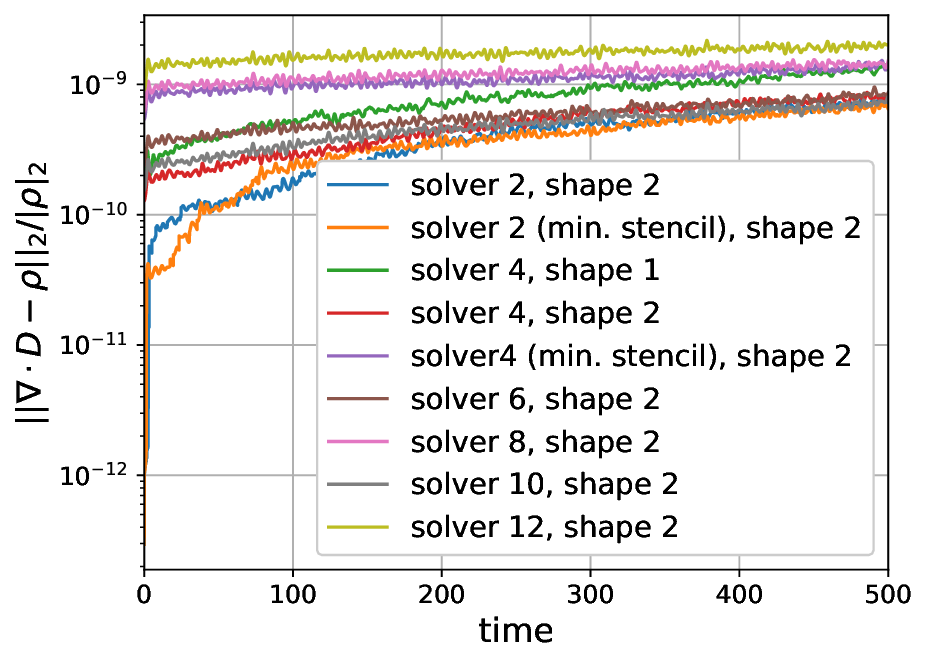}
  \caption{Conservation of Gauss' law}
  \label{fig:gauss}
\end{subfigure}
\begin{subfigure}{.49\textwidth}
  \centering
  \includegraphics[width=\linewidth]{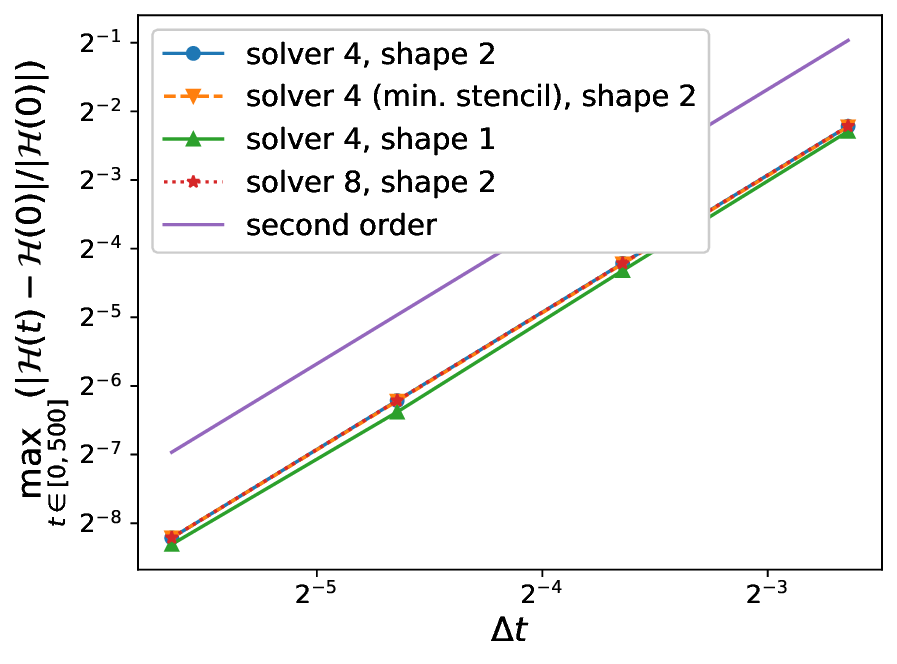}
  \caption{Energy conservation}
  \label{fig:energy}
\end{subfigure}
\caption{Visualization of the conservation properties of the code for the Weibel instability test case.}
\label{fig:weibel}
\end{figure}

\subsubsection{Degree of the shape function}

In this series of experiments we want to demonstrate the influence of the degree of the spline that is chosen as a shape function. We fix the other parameters as follows: rev{fourth} order Maxwell solver, time step of size $\Delta t = 0.02$. For the grid and particle resolution, we consider two sets of experiments: in the first set we use a coarse grid of $7\times 7 \times 7$ grid points with 2000 particles per cell. On the other hand, we use a finer grid of $15 \times 15 \times 15$ grid points but with only 200 particles-per-cell. Figure \ref{fig:order_shape_function} shows the magnetic energy curves for the two sets of resolution with varying degree of the shape function spline. Figure \ref{fig:order_shape_777} shows that the numerical results slightly damp the growth rate for all degrees of the smoothing spline. When fitting a growth rate to the data in the time interval $[100,200]$, the fitted growth rates are decreasing slightly from $0.0240$ for degree 1 to $0.0226$ for degree 4 compared to the analytical growth rate of $0.02784$, indicating a slight advantage of lower order splines in this case. 
In Figure \ref{fig:order_shape_151515}, on the other hand, we see that for this very low particle resolution a smoothing is necessary and higher degree splines yield a better result. We can conclude that particle smoothing by higher degree shape functions is necessary when the number of particles-per-cell is low, i.e. a high numerical noise is present. On the other hand, the smoothing introduced by higher degree shape functions can yield a damping of the growth rate which is why at high particle resolution it is preferable to have localized shape functions.

\begin{figure}[ht]
\begin{subfigure}{.49\textwidth}
  \centering
  \includegraphics[width=\linewidth]{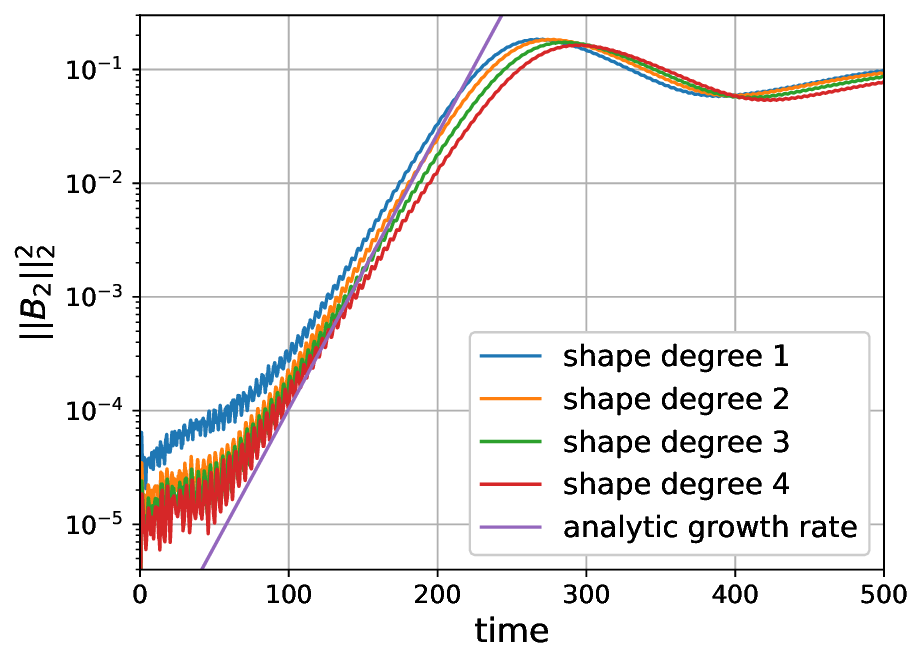}  
  \caption{Grid: $7\times 7 \times 7$, 2000 ppc}
  \label{fig:order_shape_777}
\end{subfigure}
\begin{subfigure}{.49\textwidth}
  \centering
  \includegraphics[width=\linewidth]{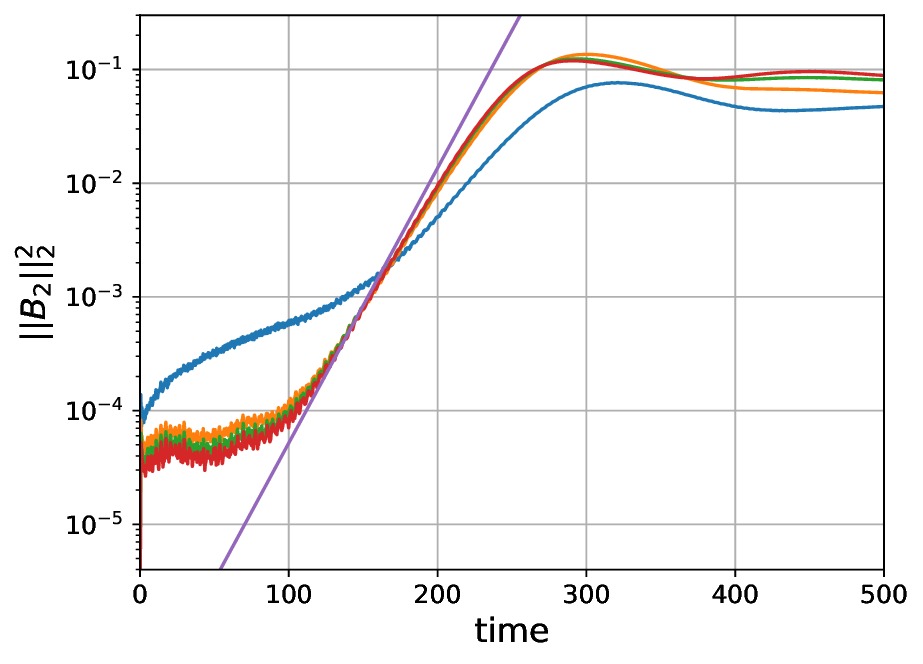}  
  \caption{Grid: $15\times 15 \times 15$, 200 ppc}
  \label{fig:order_shape_151515}
\end{subfigure}
\caption{Magnetic energy curves (second component) for various degree of the shape function (see legend) for two different resolutions (see subcaptions).}
\label{fig:order_shape_function}
\end{figure}

\section{Conclusion}

We have extended in this paper the framework of the geometric electromagnetic Particle In Cell method that has initially been developed in the context of Finite Element Exterior Calculus to mimetic Finite Differences on dual staggered grids. We defined an approximation of the electromagnetic fields on both meshes. The Amp\`ere and Faraday equations hold exactly each on one grid with one representation of the fields. The approximation order comes from a discrete Hodge operator that maps the fields on one grid to the corresponding fields on the other grid. These Hodge operators were obtained by using sliding Lagrange interpolation and the related histopolations.

We introduced an action principle using appropriate Finite Difference degrees of freedom and where the particle mesh coupling term appears in the form of reduction operators which are dual to the discrete potentials. 
The variations of this action principle yield a semi-discrete noncanonical Hamiltonian system, for which we could prove that a nontrivial Jacobi identity was verified. We also showed that the discrete divergence constraint for Maxwell's equations are automatically satisfied for all times provided they are satisfied at the initial time as they are Casimir invariants of the discrete Poisson bracket.

The new scheme was tested on relevant numerical experiments highlighting in particular, the conservation properties and the advantages of high order Maxwell solvers and spline shape functions. 

Our construction of mimetic finite differences and the structure-pre\-ser\-ving discretization on staggered grid provides a theoretical framework to show that existing finite difference codes based on the Yee scheme and Gauss law conserving current deposition are geometric and offers the possibility to extend them to high-order geometric discretization methods. Moreover, it enabled us to build our formulation in the AMReX framework and leverage in the future not only their high-performance particle module but also features like adaptive mesh refinement and embedded boundaries. Compared to the finite element version of the GEMPIC algorithm, no matrices need to be inverted in the propagation, since with the choice of $\tilde{\arrD}$ and $\arrB$ for the degrees of freedom, Hodge inversion only appears in the stationary Poisson equation, that only needs to be solved at the initial time.

\textit{Acknowledgement:} We would like to thank Martin Campos Pinto for valuable discussions and all the contributors to the GEMPICX code. This version of the manuscript contains corrections compared to the published version for which Martin  is also gratefully acknowledged.\\This work has been carried out within the framework of the EUROfusion Consortium, funded by the European Union via the Euratom Research and Training Programme (Grant Agreement No 101052200 -- EUROfusion). Views and opinions expressed are however those of the authors only and do not necessarily reflect those of the European Union or the European Commission. Neither the European Union nor the European Commission can be held responsible for them.
KK acknowledges funding from the Deutsche Forschungsgemeinschaft (DFG) via the Collaborative Research Center SFB1491 Cosmic Interacting Matters---From Source to Signal.

\bibliographystyle{abbrv}

\end{document}